\documentclass[11pt,a4paper]{amsart}		

\usepackage[utf8]{inputenc}					

\usepackage{lmodern}						
\usepackage[english]{babel}					
\usepackage[weather]{ifsym}					

\usepackage[breaklinks]{hyperref}

\usepackage[dvipsnames]{xcolor}
\hypersetup{
    colorlinks,
    linkcolor={red!50!black},
    citecolor={blue!50!black},
    urlcolor={blue!80!black}
}

\usepackage{amsmath,amsfonts,amssymb,amsthm}

\usepackage[left=3cm,right=3cm,top=3cm,bottom=3cm]{geometry}

\usepackage{mathtools}						
\mathtoolsset{showonlyrefs,showmanualtags}	

\theoremstyle{plain}
\newtheorem{theorem}{Theorem}
\newtheorem{prop}[theorem]{Proposition}
\newtheorem{cor}[theorem]{Corollary}
\newtheorem{lemma}[theorem]{Lemma}
\theoremstyle{definition}
\newtheorem{definition}{Definition}
\theoremstyle{remark}
\newtheorem{rmk}{Remark}
\newtheorem{example}{Example}

\DeclareMathOperator{\spn}{\mathrm{span}}
\DeclareMathOperator{\rank}{\mathrm{rank}}
\DeclareMathOperator{\tr}{\mathrm{Trace}}
\DeclareMathOperator{\diam}{diam}

\newcommand{\R}{\mathbb{R}}					
\newcommand{\C}{\mathbb{C}}					
\renewcommand{\H}{\mathbb{H}}				
\newcommand{\distr}{\mathcal{D}}			
\newcommand{\ver}{\mathcal{V}}				
\renewcommand{\theta}{\vartheta}			
\renewcommand{\epsilon}{\varepsilon}		
\newcommand{\euler}{\mathfrak{e}}			
\newcommand{\vis}{\text{\tiny\HalfSun}}		
\newcommand{\liousurf}{{\mu}}
\newcommand{\dive}{\mathrm{div}}			
\renewcommand{\d}{\mathsf{d}}				

\title[A sub-Riemannian Santal\'o formula]{A sub-Riemannian Santal\'o formula with applications to isoperimetric inequalities and first Dirichlet eigenvalue of hypoelliptic operators}

\author{Dario Prandi$^1$}
\address{$^1$CNRS, L2S, CentraleSup\'elec, Gif-sur-Yvette, France.}
\email{\href{mailto:dario.prandi@l2s.centralesupelec.fr}{dario.prandi@l2s.centralesupelec.fr}}

\author{Luca Rizzi$^{2,3}$}
\address{$^2$Univ. Grenoble Alpes, CNRS, Institut Fourier, F-38000 Grenoble, France}
\address{$^3$Inria, team GECO \& CMAP, \'Ecole Polytechnique, CNRS, Universit\'e Paris-Saclay, Palaiseau, France (past institution)}
\email{\href{mailto:luca.rizzi@univ-grenoble-alpes.fr}{luca.rizzi@univ-grenoble-alpes.fr}}

\author{Marcello Seri$^4$}
\address{$^4$Department of Mathematics and Statistics, University of Reading, Reading, UK}
\email{\href{mailto:m.seri@ucl.ac.uk}{m.seri@ucl.ac.uk}}


\begin{document}

\begin{abstract}
In this paper we prove a sub-Riemannian version of the classical Santal\'o formula: a result in integral geometry that describes the intrinsic Liouville measure on the unit cotangent bundle in terms of the geodesic flow. Our construction works under quite general assumptions, satisfied by any sub-Riemannian structure associated with a Riemannian foliation with totally geodesic leaves (e.g.\ CR and QC manifolds with symmetries), any Carnot group, and some non-equiregular structures such as the Martinet one. A key ingredient is a ``reduction procedure'' that allows to consider only a simple subset of sub-Riemann\-ian geodesics.

As an application, we derive isoperimetric-type and ($p$-)Hardy-type inequalities for a compact domain $M$ with piecewise $C^{1,1}$ boundary, and a universal lower bound for the first Dirichlet eigenvalue $\lambda_1(M)$ of the sub-Laplacian,
\begin{equation}
\lambda_1(M)  \geq \frac{k \pi^2}{L^2},
\end{equation}
in terms of the rank $k$ of the distribution and the length $L$ of the longest reduced sub-Riemannian geodesic contained in $M$. All our results are sharp for the sub-Riemannian structures on the hemispheres of the complex and quaternionic Hopf fibrations:
\begin{equation}
\mathbb{S}^1\hookrightarrow \mathbb{S}^{2d+1} \xrightarrow{p} \mathbb{CP}^d, \qquad \mathbb{S}^3\hookrightarrow \mathbb{S}^{4d+3} \xrightarrow{p} \mathbb{HP}^d, \qquad d \geq 1,
\end{equation}
where the sub-Laplacian is the standard hypoelliptic operator of CR and QC geometries, $L = \pi$ and $k=2d$ or $4d$, respectively. 
\end{abstract}

\maketitle

\section{Introduction and results}

Let $(M,g)$ be a compact connected Riemannian manifold with boundary $\partial M$.
Santal\'o formula \cite{chavelbook,santalobook} is a classical result in integral geometry that describes the Liouville measure $\mu$ of the unit tangent bundle $UM$ in terms of the geodesic flow $\phi_t:UM\to UM$. Namely, for any measurable function $F:UM\to \R$ we have
\begin{equation}
	\label{eq:santalo-intro}
	\int_{U^\vis M} F \,\liousurf = \int_{\partial M} \left[\int_{U^+_q\partial M}  \left( \int_0^{\ell(v)}  F ( \phi_t(v))  dt\right)  g(v,{\mathbf{n}_q}) \eta_q(v)\right] \sigma(q),
\end{equation}
where $\sigma$ is the surface form on $\partial M$ induced by the inward pointing normal vector $\mathbf n$, $\eta_q$ is the Riemannian spherical measure on $U_qM$, $U^+_q\partial M$ is the set of inward pointing unit vectors at $q\in \partial M$ and $\ell(v)$ is the exit length of the geodesic with initial vector $v$.
Finally, $U^\vis M\subseteq UM$ is the visible set, i.e.\ the set of unit vectors that can be reached via the geodesic flow starting from points on $\partial M$.

In the Riemannian setting,~\eqref{eq:santalo-intro} allows to deduce some very general isoperimetric inequalities and Dirichlet eigenvalues estimates for the Laplace-Beltrami operator as showed by Croke in the celebrated papers \cite{croke-iso, croke-iso-4D, croke-lower}.

The extension of~\eqref{eq:santalo-intro} to the sub-Riemannian setting and its consequences are not straightforward for a number of reasons. Firstly, in sub-Riemannian geometry the geodesic flow is replaced by a degenerate Hamiltonian flow on the cotangent bundle. Moreover, the unit cotangent bundle (the set of covectors with unit norm) is not compact, but rather has the topology of an infinite cylinder. Finally, in sub-Riemannian geometry there is not a clear agreement on which is the ``canonical'' volume, generalizing the Riemannian measure. Another aspect to consider is the presence of characteristic points on the boundary.

In this paper we extend~\eqref{eq:santalo-intro} to the most general class of sub-Riemann\-ian structures for which Santal\'o formula makes sense. As an application we deduce Hardy-like inequalities, sharp universal estimates on the first Dirichlet eigenvalue of the sub-Laplacian and sharp isoperimetric-type inequalities.

To our best knowledge, a sub-Riemann\-ian version of \eqref{eq:santalo-intro} appeared only in~\cite{pansu} for the three-dimensional Heisenberg group, and more recently in \cite{montefalcone} for Carnot groups, where the natural global coordinates allow for explicit computations. As far as other sub-Riemannian structures are concerned, Santal\'o formula is an unexplored technique with potential applications to different settings, including CR (Cauchy-Rie\-mann) and QC (quaternionic contact) geometry, Riemannian foliations, and Carnot groups. 

\subsection{Setting and examples}
Let $(N,\distr,g)$ be a sub-Riemannian manifold of dimension $n$, where $\distr\subseteq TN$ is a distribution that satisfies the bracket-generating condition and $g$ is a smooth metric on $\distr$. Smooth sections $X \in \Gamma(\distr)$ are called \emph{horizontal}. We consider a compact $n$-dimensional submanifold $M \subset N$ with boundary $\partial M \neq \emptyset$. 

If $(N,\distr,g)$ is Riemannian, we equip it with its Riemannian volume $\omega_R$. In the genuinely sub-Riemannian case we fix any smooth volume form $\omega$ on $M$ (or a density if $M$ is not orientable). In any case, the surface measure $\sigma = \iota_{\mathbf{n}} \omega$ on $\partial M$ is given by the contraction with the horizontal unit normal $\mathbf{n}$ to $\partial M$. For what concerns the regularity of the boundary, we assume only that $\partial M$ is piecewise $C^{1,1}$. (See Remark~\ref{rmk:noH0} for the Lipschitz case.)

A central role is played by sub-Riemannian geodesics, i.e., curves tangent to $\distr$ that locally minimize the sub-Riemannian distance between endpoints. In this setting, the geodesic flow\footnote{Abnormal geodesics are allowed, but strictly abnormal ones, not given by the Hamiltonian flow on the cotangent bundle, do not play any role in our  construction.} is a natural Hamiltonian flow $\phi_t:T^*M \to T^*M$ on the cotangent bundle, induced by the Hamiltonian function $H \in C^\infty(T^*M)$. The latter is a non-negative, degenerate, quadratic form on the fibers of $T^*M$ that contains all the information on the sub-Riemannian structure. Length-parametrized geodesics are characterized by an initial covector $\lambda$ in the unit cotangent bundle $U^*M = \{\lambda \in T^*M \mid 2H(\lambda) =1 \}$. 

A key ingredient for most of our results is the following \emph{reduction procedure}. Fix a transverse sub-bundle $\ver\subset TM$ such that $TM=\distr\oplus \ver$. We define the reduced cotangent bundle $T^*M^\mathsf{r}$ as the set of covectors annihilating $\ver$. On $T^*M^\mathsf{r}$ we define a reduced Liouville volume $\Theta^\mathsf{r}$, which depends on the choice of the volume $\omega$ on $M$. These must satisfy the following stability hypotheses:
\begin{itemize}
\item[(\textbf{H1})] The bundle $T^*M^\mathsf{r}$ is invariant under the Hamiltonian flow $\phi_t$;
\item[(\textbf{H2})] The reduced Liouville volume is invariant, i.e. $\mathcal{L}_{\vec{H}}\Theta^\mathsf{r} = 0$.
\end{itemize}
This allows to replace the non-compact $U^*M$ with a compact slice $U^*M^\mathsf{r}:= U^*M \cap T^*M^\mathsf{r}$, equipped with an invariant measure (see Section~\ref{sec:redsant}). These hypotheses are verified for:
\begin{itemize}
\item any Riemannian structure, equipped with the Riemannian volume;
\item any sub-Riemannian structure associated with a Riemannian foliation with totally geodesic leaves, equipped with the Riemannian volume. These includes contact, CR, QC structures with transverse symmetries, and also some non-equiregular structures as the Martinet one on $\R^3$. See Section~\ref{s:foliations};
\item any left-invariant sub-Riemannian structure on a Carnot group\footnote{We stress that Carnot groups are not Riemannian foliations if their step is $> 2$.}, equipped with the Haar volume, see Section~\ref{s:Carnot}.
\end{itemize}

An interesting example, coming from CR geometry, is the complex Hopf fibration (CHF)
\begin{equation}
\mathbb{S}^1\hookrightarrow \mathbb{S}^{2d+1} \xrightarrow{p} \mathbb{CP}^d, \qquad d \geq 1,
\end{equation}
where $\distr := (\ker p_*)^\perp$ is the orthogonal complement of the kernel of the differential of the Hopf map w.r.t.\ the round metric on $\mathbb{S}^{2d+1}$, and the sub-Riemannian metric $g$ is the restriction to $\distr$ of the round one. Another interesting structure, coming from QC geometry and with corank $3$, is the quaternionic Hopf fibration (QHF)
\begin{equation}
\mathbb{S}^3\hookrightarrow \mathbb{S}^{4d+3} \xrightarrow{p} \mathbb{HP}^d, \qquad d \geq 1,
\end{equation}
where $\mathbb{HP}^d$ is the quaternionic projective space of real dimension $4d$ and the sub-Riemann\-ian structure on $\mathbb{S}^{4d+3}$ is defined similarly to its complex version.

\subsection{Sub-Riemannian Santal\'o formulas}

Consider a sub-Riemann\-ian geodesic $\gamma(t)$ with initial covector $\lambda \in U^*M$. The \emph{exit length} $\ell(\lambda) \in [0,+\infty)$ is the length after which $\gamma$ leaves $M$ by crossing $\partial M$. Similarly, $\tilde{\ell}(\lambda)$ is the minimum between $\ell(\lambda)$ and the cut length $c(\lambda)$. That is, after length $\tilde{\ell}(\lambda)$ the geodesic either loses optimality or leaves $M$.

\begin{figure}
\includegraphics[scale=1]{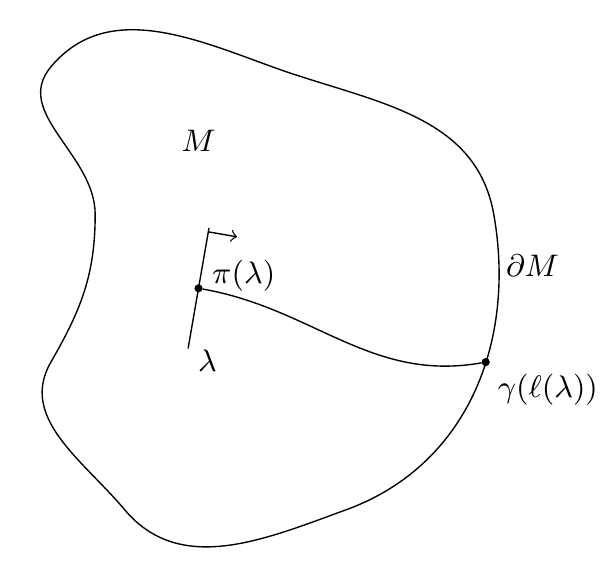}
\hfill
\includegraphics[scale=1]{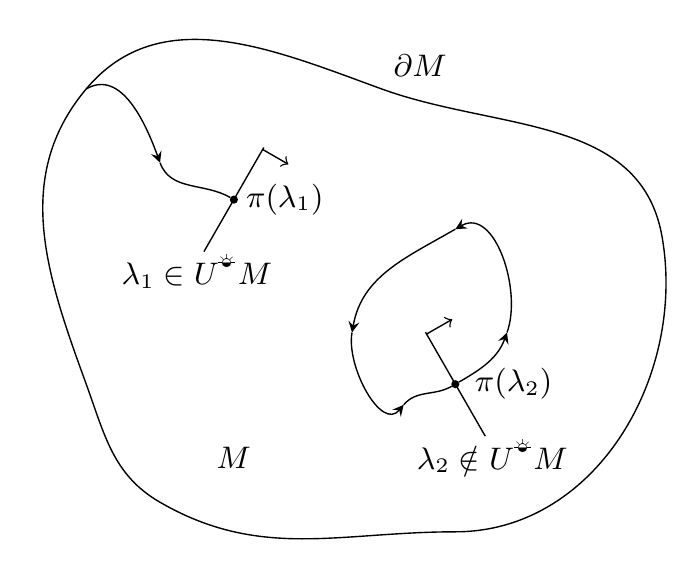}
\caption{Exit length (left) and visible set (right). Covectors are represented as hyperplanes, the arrow shows the direction of propagation of the associated geodesic for positive time.}\label{f:fig1}
\end{figure}

The \emph{visible unit cotangent bundle} $U^\vis M \subset U^*M$ is the set of unit covectors $\lambda$ such that $\ell(-\lambda) < +\infty$. (See Fig.~\ref{f:fig1}.) Analogously, the \emph{optimally visible unit cotangent bundle} $\tilde U^\vis M$ is the set of unit covectors such that $\tilde{\ell}(-\lambda) < +\infty$.

For any non-characteristic point $q \in \partial M$, we have a well defined inner pointing unit horizontal vector $\mathbf{n}_q \in \distr_q$, and $U^+_q\partial M \subset U_q^*M$ is the set of initial covectors of geodesics that, for positive time, are directed toward the interior of $M$.

As anticipated, we do not consider all the length-para\-met\-rized ge\-odesics, i.e. all initial covectors $\lambda \in U_q^* M \simeq \mathbb{S}^{k-1} \times \R^{n-k}$, but a \emph{reduced subset} $U_q^* M^\mathsf{r} \simeq \mathbb{S}^{k-1}$. In the following the suffix $\mathsf{r}$ always denotes the intersection with the reduced unit cotangent bundle $U^*M^\mathsf{r}$. We stress the critical fact that $U^*M^\mathsf{r}$ is compact, while $U^*M$ never is, except in the Riemannian setting where the reduction procedure is trivial. With these basic definitions at hand, we are ready to state the sub-Riemannian Santal\'o formulas.

\begin{theorem}[Reduced Santal\'o formulas]
	\label{t:santalo-reduced-intro}
	The visible set $U^\vis M^\mathsf{r}$ and the optimally visible set $\tilde{U}^\vis M^\mathsf{r}$ are measurable. For any measurable function $F:U^*M^\mathsf{r} \to \R$ we have
	\begin{align}
		\label{eq:santalo-reduced-intro}
		\int_{U^\vis M^\mathsf{r}} F \,\liousurf^\mathsf{r}
			& = \int_{\partial M} \left[\int_{U_q^+\partial M^\mathsf{r}}  \left( \int_0^{\ell(\lambda)}  F ( \phi_t(\lambda))  dt\right)  \langle \lambda,{\mathbf{n}_q}\rangle \eta^\mathsf{r}_q(\lambda)\right] \sigma(q), \\
		\label{eq:santalo-tilde-reduced-intro} 
		\int_{\tilde{U}^\vis M^\mathsf{r}} F \,\liousurf^\mathsf{r}
			& = \int_{\partial M} \left[\int_{U_q^+\partial M^\mathsf{r}}  \left( \int_0^{\tilde{\ell}(\lambda)}  F ( \phi_t(\lambda))  dt\right)  \langle \lambda,{\mathbf{n}_q}\rangle \eta^\mathsf{r}_q(\lambda)\right] \sigma(q).
\end{align}
\end{theorem}

In~\eqref{eq:santalo-reduced-intro}-\eqref{eq:santalo-tilde-reduced-intro}, $\mu^\mathsf{r}$ is a reduced invariant Liouville measure on $U^*M^\mathsf{r}$, $\eta_q^\mathsf{r}$ is an appropriate smooth measure on the fibers $U_q^*M^\mathsf{r}$ and $\langle\lambda,\cdot\rangle$ denotes the action of covectors on vectors. Indeed both include the Riemannian case, where the reduction procedure is trivial and $U^*M \simeq UM$ since the Hamiltonian is not degenerate.

\begin{rmk}
Hypotheses (\textbf{H1}) and (\textbf{H2}) are essential for the reduction procedure. An unreduced version of Theorem~\ref{t:santalo-reduced-intro} holds for any volume $\omega$ and with no other assumptions but the Lipschitz regularity of $\partial M$ (see Theorem~\ref{t:santalo} and Remark~\ref{rmk:noH0}).  However, the consequences we present do not hold a priori, as their proofs rely on the summability of certain functions on $U^*M^\mathsf{r}$, generally false on the non-compact $U^*M$.
\end{rmk}

\subsection{Hardy-type inequalities}
\label{sec:hardy-type-intro}

For any $f \in C^\infty(M)$, let $\nabla_H f \in \Gamma(\distr)$ be the \emph{horizontal gradient}: the horizontal direction of steepest increase of $f$. It is defined via the identity
\begin{equation}\label{eq:horgrad}
g(\nabla_H f, X) = df(X), \qquad \forall X \in \Gamma(\distr).
\end{equation}

Consider all length-parametrized sub-Riemannian geodesic passing through a point $q \in M$, with covector $\lambda \in U^*_qM$. Set $L(\lambda) := \ell(\lambda) +\ell(-\lambda)$; this is the length of the maximal geodesic that passes through $q$ with covector $\lambda$.

\begin{prop}[Hardy-like inequalities]\label{p:hardylike-intro}
For any $f\in C_0^\infty(M)$ it holds
\begin{align}
\label{eq:hardy1-intro}  \int_M |\nabla_H f|^2 \omega &\ge
    \frac{k\pi^2}{| \mathbb{S}^{k-1}|} \int_{M} \frac{f^2}{R^2} \omega, \\
\label{eq:hardy2-intro}  \int_M |\nabla_H f|^2 \omega &\ge
    \frac{k}{4| \mathbb{S}^{k-1}|} \int_{M} \frac{f^2}{r^2} \omega,
\end{align}
where $k = \rank \distr$ and $r,R: M \to \R$ are:
\begin{equation}
\frac{1}{R^2(q)}  := \int_{U_q^* M^\mathsf{r}}\frac{1}{L^2}\eta^\mathsf{r}_q, \qquad \frac{1}{r^2(q)}  := \int_{U_q^* M^\mathsf{r}}\frac{1}{\ell^2}\eta^\mathsf{r}_q, \qquad \forall q \in M.
\end{equation}
\end{prop}
We observe that $r$ is the \emph{harmonic mean distance from the boundary} defined in \cite{Davies-review}. One can also consider the following generalization of Proposition \ref{p:hardylike-intro} for $L^p(M,\omega)$ norms.

\begin{prop}[$p$-Hardy-like inequality] \label{p:p-hardy-intro}
    Let $p>1$ and $f\in C_0^\infty(M)$. Then
 \begin{align}
\int_M |\nabla_H f|^p \omega &\ge
    \pi_p^p \, C_{p,k} \int_{M} \frac{|f|^p}{R^p} \omega,  \label{eq:p-hardy1-intro}\\
\int_M |\nabla_H f|^p \omega &\ge
    \left(\frac{p-1}{p}\right)^p C_{p,k} \int_{M} \frac{|f|^p}{r^p} \omega, \label{eq:p-hardy2-intro}
\end{align}
where $k = \rank \distr$, the constants $\pi_p$ and $C_{p,k}$ are
\begin{equation}
\pi_p = \frac{2\pi (p-1)^{1/p}}{p \, \sin(\pi/p)}, \qquad C_{p,k} = \frac{k}{|\mathbb{S}^{k-1}|} \frac{\sqrt{\pi}\,\Gamma(\tfrac{k+p}{2})}{2\Gamma(\tfrac{1+p}{2})\Gamma(\tfrac{k}{2}+1)},
\end{equation}
and $r^p,R^p: M \to \R$ are
\begin{equation}
\frac{1}{R^p(q)}  := \int_{U_q^* M^\mathsf{r}}\frac{1}{L^p}\eta^\mathsf{r}_q, \qquad \frac{1}{r^p(q)}  := \int_{U_q^* M^\mathsf{r}}\frac{1}{\ell^p}\eta^\mathsf{r}_q, \qquad \forall q \in M.
\end{equation}
\end{prop}

\subsection{Lower bound for the first Dirichlet eigenvalue}
\label{sec:spectral-gap-intro}

For any given smooth volume $\omega$, a fundamental operator in sub-Riemannian geometry is the sub-Laplacian $\Delta_\omega$, playing the role of the Laplace-Beltrami operator in Riemannian geometry. Under the bracket-generating condition, this is an hypoelliptic operator on $L^2(M,\omega)$. Its principal symbol is (twice) the Hamiltonian, thus the Dirichlet spectrum of $-\Delta_\omega$ on the compact manifold $M$ is positive and discrete. We denote it
\begin{equation}
0 < \lambda_1(M) \leq \lambda_2(M) \leq \ldots.
\end{equation}
As a consequence of Proposition~\ref{p:hardylike-intro} and the min-max principle, we obtain a universal lower bound for the first Dirichlet eigenvalue $\lambda_1(M)$ on the given domain. Here by universal we mean an estimate not requiring any assumption on curvature or capacity.

\begin{prop}[Universal spectral lower bound] \label{p:lowerbound-intro}
Let $L = \sup_{\lambda \in U^*M^{\mathsf r}} L(\lambda)$ be the length of the longest reduced geodesic contained in $M$. Then, letting $k = \rank \distr$,
\begin{equation}\label{eq:lambda1-intro}
\lambda_1(M) \ge  \frac{k\pi^2}{L^2},
\end{equation}
where we set the r.h.s. to $0$ if $L= +\infty$.
\end{prop}
\begin{rmk}
  In~\eqref{eq:lambda1-intro}, $L$ cannot be replaced by the sub-Riemannian diameter, as $M$ might contain very long (non-minimizing) geodesics, for example closed ones, and $L = +\infty$. See Appendix~\ref{sec:closed-geodesics} for more details.
\end{rmk}
In the Riemannian case, as noted by Croke, we attain equality in~\eqref{eq:lambda1-intro} when $M$ is the hemisphere of the Riemannian round sphere. We prove the following extension to the sub-Riemannian setting.

\begin{prop}[Sharpness of the eigenvalue lower bound]\label{p:sharpness-lowerbound-intro}
In Proposition~\ref{p:lowerbound-intro}, in the following cases we have equality, for all $d \geq 1$:
	\begin{itemize}
	\item[(i)] the hemispheres $\mathbb{S}^d_+$ of the Riemannian round sphere $\mathbb{S}^d$;
	\item[(ii)] the hemispheres $\mathbb{S}^{2d+1}_+$ of the sub-Riemannian complex Hopf fibration $\mathbb{S}^{2d+1}$;
	\item[(iii)] the hemispheres $\mathbb{S}^{4d+3}_+$  of the sub-Riemannian quaternionic Hopf fibration $\mathbb{S}^{4d+3}$;
	\end{itemize}
all equipped with the Riemannian volume of the corresponding round sphere. In all these cases,  $L = \pi$ and $\lambda_1(M) = d$, $2d$ or $4d$, respectively. Moreover, the associated eigenfunction is $\Psi = \cos(\delta)$, where $\delta$ is the Riemannian distance from the north pole.
\end{prop}

\begin{rmk}
The Riemannian volume of the sub-Riemannian Hopf fibrations coincides, up to a constant factor, with their Popp volume \cite{BR-Popp,montgomerybook}, an intrinsic smooth measure in sub-Riemannian geometry. 
This is proved for 3-Sasakian structures (including the QHF) in \cite[Prop. 34]{RS-fatcomparison} and can be proved exactly in the same way for Sasakian structures (including the CHF) using the explicit formula for Popp volume of \cite{BR-Popp}. For the case (i) $\Delta_\omega$ is the Laplace-Beltrami operator. For the cases (ii) and (iii) $\Delta_\omega$ is the standard sub-Laplacian of CR and QC geometry, respectively.
\end{rmk}

In principle, $L$ can be computed when the reduced geodesic flow is explicit. This is the case for Carnot groups, where reduced geodesics passing through the origin are simply straight lines (they fill a $k$-plane for rank $k$ Carnot groups). It turns out that, in this case, $L = \diam_H(M)$ (the \emph{horizontal diameter}, that is the diameter of the set $M$ measured through left-translations of the aforementioned straight lines).  Thus~\eqref{eq:lambda1-intro} gives an easily computable lower bound for the first Dirichlet eigenvalue in terms of purely metric quantities.
\begin{cor}\label{c:carnot1-intro}
Let $M$ be a compact $n$-dimensional submanifold with piecewise $C^{1,1}$ boundary of a Carnot group of rank $k$, with the Haar volume. Then,
\begin{equation}\label{eq:lambda1carnot-intro}
\lambda_1(M) \ge \frac{k\pi^2}{\diam_H(M)^2},
\end{equation}
where $\diam_H(M)$ denotes the horizontal diameter of $M$.
\end{cor}
In particular, if $M$ is the metric ball of radius $R$, we obtain $\lambda_1(M) \geq k \pi^2/ (2R)^2$. Clearly \eqref{eq:lambda1carnot-intro} is not sharp, as one can check easily in the Euclidean case.

\subsection{Isoperimetric-type inequalities}
\label{sec:isoperimetric-intro}

In this section we relate the sub-Riemann\-ian area and perimeter of $M$ with some of its geometric properties. Since $M$ is compact, the sub-Riemannian diameter $\diam(M)$ can be characterized  as the length of the longest optimal geodesic contained in $M$. Analogously, the reduced sub-Riemannian diameter $\diam^\mathsf{r}(M)$ is the length of the longest reduced optimal geodesic contained in $M$. Indeed $\diam^\mathsf{r}(M) \leq \diam(M)$.

\begin{figure}
\includegraphics[scale=1]{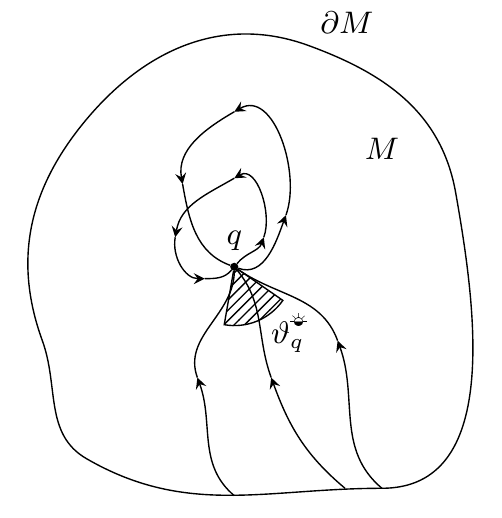}
\caption{Visibility angle on a 2D Riemannian manifold. Only the geodesics with tangent vector in the dashed slice go to $\partial M$.}\label{f:fig2}
\end{figure}
Consider all reduced geodesics passing through $q \in M$ with covector $\lambda$. Some of them originate from the boundary $\partial M$, that is $\ell(-\lambda) < +\infty$; others do not, i.e. $\ell(-\lambda) = +\infty$. The relative ratio of  the lengths of these two types of geodesics (w.r.t. an appropriate measure on $U_q^* M^\mathsf{r}$) is called the \emph{visibility angle} $\theta^\vis_q  \in [0,1]$ at $q$ (see Definition~\ref{def:vis-angle}). Roughly speaking, if $\theta^\vis_q = 1$ then any geodesic passing through $q$ will hit the boundary and, on the opposite, if it is equal to $0$ then $q$ is not visible from the boundary (see Fig.~\ref{f:fig2}). Similarly, we define the \emph{optimal visibility angle} $\tilde\theta^\vis_q$ by replacing $\ell(-\lambda)$ with $\tilde{\ell}(-\lambda)$. Finally, the \emph{least visibility angle} is $\theta^\vis := \inf_{q \in M} \theta^\vis_q$, and similarly for the \emph{least optimal visibility angle} $\tilde\theta^\vis := \inf_{q \in M} \tilde\theta^\vis_q$.

\begin{prop}[Isoperimetric-type inequalities]\label{p:type1iso-intro}
Let $\ell := \sup\{\ell(\lambda)  \mid \lambda \in U_q^*M^{\mathsf r},\, q \in \partial M\}$ be the length of the longest reduced geodesic contained in $M$ starting from the boundary $\partial M$. Then
	\begin{equation}\label{eq:type1iso-intro}
		\frac{\sigma(\partial M)}{\omega(M)} \ge C  \frac{\theta^\vis}{\ell} 
		\qquad \text{and} \qquad
		\frac{\sigma(\partial M)}{\omega(M)} \ge    C \frac{\tilde\theta^\vis}{\diam^\mathsf{r}(M)} ,
	\end{equation}
	where $C = 2\pi|\mathbb S^{k-1}|/|\mathbb S^{k}|$ and we set the r.h.s.\ to $0$ if $\ell = +\infty$.
\end{prop}

The equality in~\eqref{eq:type1iso-intro} holds for the hemisphere of the Riemannian round sphere, as pointed out in \cite{croke-iso}. We have the following generalization to the sub-Riemannian setting.
\begin{prop}[Sharpness of isoperimetric inequalities]\label{p:sharp-iso-intro}
In Proposition~\ref{p:type1iso-intro}, in the following cases we have equality, for all $d \geq 1$:
	\begin{itemize}
	\item[(i)] the hemispheres $\mathbb{S}^d_+$ of the Riemannian round sphere $\mathbb{S}^d$;
	\item[(ii)] the hemispheres $\mathbb{S}^{2d+1}_+$ of the sub-Riemannian complex Hopf fibration $\mathbb{S}^{2d+1}$;
	\item[(iii)] the hemispheres $\mathbb{S}^{4d+3}_+$  of the sub-Riemannian quaternionic Hopf fibration $\mathbb{S}^{4d+3}$;
	\end{itemize}
where $\omega$ is the Riemannian volume of the corresponding round sphere. In all these cases $\theta^\vis = \tilde\theta^\vis = 1$ and $\ell= \diam^\mathsf{r}(M) = \pi$.
\end{prop}

We can apply Proposition~\ref{p:type1iso-intro} to Carnot groups equipped with the Haar measure. In this case $\theta^\vis = \tilde\theta^\vis = 1$ and $\ell = \diam^\mathsf{r}(M) = \diam_H(M)$. Moreover, $\omega$ is the Lebesgue volume of $\R^n$ and $\sigma$ is the associated perimeter measure of geometric measure theory \cite{Capognabook}.
\begin{cor}\label{c:carnot2-intro}
Let $M$ be a compact $n$-dimensional submanifold with piecewise $C^{1,1}$ boundary of a Carnot group of rank $k$, with the Haar volume. Then,
\begin{equation}
\frac{\sigma(\partial M)}{\omega(M)} \geq \frac{2\pi |\mathbb{S}^{k-1}|}{|\mathbb{S}^k|\diam_H(M)},
\end{equation}
where $\diam_H(M)$ is the horizontal diameter of the Carnot group.
\end{cor}
This inequality is not sharp even in the Euclidean case, but it is very easy to compute the horizontal diameter for explicit domains. For example, if $M$ is the sub-Riemannian metric ball of radius $R$, then $\diam_H(M)= 2R$. 

\subsection{Remark on change of volume}

Fix a sub-Riemannian structure $(N,\distr,g)$, a compact set $M$ with piecewise $C^{1,1}$ boundary and a complement $\ver$ such that (\textbf{H1}) holds. Now assume that, for some choice of volume form $\omega$, also (\textbf{H2}) is satisfied, so that we can carry on with the reduction procedure and all our results hold. One can derive the analogous of Propositions~\ref{p:hardylike-intro}, \ref{p:p-hardy-intro}, \ref{p:lowerbound-intro}, \ref{p:type1iso-intro} for any other volume $\omega' = e^\varphi \omega$, with $\varphi \in C^\infty(M)$. In all these results, it is sufficient to multiply the r.h.s. of the inequalities by the \emph{volumetric constant} $0< \alpha \leq 1$ defined as $\alpha := \frac{\min e^\varphi}{\max e^\varphi}$, and indeed replace $\omega$ with $\omega' = e^\varphi \omega$ in Propositions~~\ref{p:hardylike-intro} and \ref{p:p-hardy-intro}, $\sigma$ with $\sigma' = e^\varphi \sigma$ in Proposition~\ref{p:type1iso-intro}, and the sub-Laplacian $\Delta_\omega$ with $\Delta_{\omega'} = \Delta_\omega + \langle d \varphi,\nabla_H \,\cdot \rangle$ in Proposition~\ref{p:lowerbound-intro}. 
Analogously, one can deal with the Corollaries~\ref{c:carnot1-intro} and \ref{c:carnot2-intro} about Carnot groups.

This remark allows, for example, to obtain results for (sub-)Riemann\-ian weighted measures. This is particularly interesting in the genuinely sub-Riemannian setting since, in some cases, the volume satisfying (\textbf{H2}) might not coincide with the intrinsic Popp one.
 
\subsection{Remark on rigidity}

The sharpness results of Propositions~\ref{p:sharpness-lowerbound-intro} and \ref{p:sharp-iso-intro} hold for hemispheres of (sub-)Riemann\-ian structures associated with Riemannian submersions of the sphere with totally geodesic fibers, which have been completely classified in \cite{Escobales}. The only case which is not covered in these propositions is the so-called octonionic Hopf fibration (OHF) $\mathbb{S}^7 \hookrightarrow \mathbb{S}^{15} \to \mathbb{OP}^1$, which to our best knowledge has not yet been studied from the sub-Riemannian point of view, and for which explicit expressions for the sub-Laplacian do not appear in the literature. It is however likely that the sharpness results of Propositions~\ref{p:sharpness-lowerbound-intro} and \ref{p:sharp-iso-intro} hold also for the hemisphere $\mathbb S_+^{15}$ of the sub-Riemannian OHF.

Finally, concerning the universal lower bound of Proposition~\ref{p:lowerbound-intro}, Croke proved the following rigidity result in the Riemannian case (see \cite[Thm. 16]{croke-iso}). As we already remarked, the lower bound \eqref{eq:lambda1-intro} is non-trivial if and only if all geodesics starting from points of $M$ hit the boundary at some finite time (i.e.\ $\theta^\vis =1$). If, furthermore, every such geodesic minimizes distance up to the point of intersection with the boundary  (i.e.\ $\tilde{\theta}^\vis =1$), then we have equality in \eqref{eq:lambda1-intro} if and only if $M$ is an hemisphere of the round sphere. See also \cite{croke-lower} for a more general rigidity result. The following question is thus natural.

\vspace{.7em}

\textbf{Open question.}
Are the hemispheres of the CHF, QHF, and possibly OHF, the only domains on compact sub-Riemannian manifolds tamed by a foliation with totally geodesic leaves (see Section~\ref{s:foliations}) where the lower bound of Proposition~\ref{p:lowerbound-intro} is attained?

\subsection{Afterwords and further developments}

Despite its broad range of applications in Riemannian geometry and its Finsler generalizations \cite{YZ-FinslerSantalo}, only a few works used Santal\'o formula in the hypoelliptic setting, all of them in the specific case of Carnot groups \cite{montefalcone,pansu} or 3D Sasakian structures \cite{CY-Iso}. It is interesting to notice that, in \cite{pansu}, Pansu was able to use Santal\'o formula in pairs with minimal surfaces to eliminate the diameter term in Corollary~\ref{c:carnot2-intro} and obtain his celebrated isoperimetric inequality. In our general setting, this is something worth investigating. 

The study of spectral properties of hypoelliptic operators is an active area of research. Many results are available for the complete spectrum of the sub-Laplacian on \emph{closed} manifolds (with no boundary conditions). We recall \cite{BB-SU2,BW-CR,BW-QHF} for the case of $\mathrm{SU}(2)$, CHF and QHF. 
Furthermore, in \cite{cht15}, one can find the spectrum of the ``flat Heisenberg case'' (a compact quotient of the Heisenberg group) together with quantum ergodicity results for 3D contact sub-Riemannian structures.
Lower bounds for the first (non-zero) eigenvalue of the sub-Laplacian on closed foliation, under curvature-like assumptions, appeared in \cite{bk14} (see also \cite{BaudoinlecturenotesIHP} for a more general statement).

Concerning the Dirichlet spectrum on Riemannian manifolds with boundary, a classical reference is \cite{chavelbook-eigen}.
In the sub-Riemannian setting, we are aware of results for the sum of Dirichlet eigenvalues \cite{Strich-sum} by Strichartz and related spectral inequalities \cite{Lep-spec} by Hannson and Laptev, both for the case of the Heisenberg group.  To our best knowledge, Proposition~\ref{p:lowerbound-intro} is the first sharp universal lower bound for the first Dirichlet eigenvalue in the sub-Riemannian setting and in particular for non-Carnot structures. 

The study of Hardy's inequalities, already in the Euclidean setting, ranges across the last century and continues to the present day (see \cite{barbatis, bm98, ekl14} and references therein). The sub-Riemannian case is more recent, for an account of the known result we mention the works for Carnot groups of Capogna, Danielli and Garofalo (see e.g. \cite{cdg94, dg11}).

Poincar\'e inequalities are strictly connected to Hardy's ones. On this subject the literature is again huge, we already mentioned the works of Croke and Derdzinski concerning the Riemannian case \cite{croke-iso, croke-iso-4D, croke-lower}. Finally, see \cite{iv15} for results on CR and QC manifolds under Ricci curvature assumptions in the spirit of the Lichnerowicz-Obata theorem.


In this paper we focused mostly on foliations, where our results are sharp. For Carnot groups, Corollaries \ref{c:carnot1-intro} and \ref{c:carnot2-intro} appeared in \cite{montefalcone} and are not sharp. Let us consider for simplicity the 3D Heisenberg group, with coordinates $(x,y,z) \in \R^3$. A relevant class of domains for the Dirichlet eigenvalues problem are the ``Heisenberg cubes'' $[0,\varepsilon] \times [0,\varepsilon] \times [0,\varepsilon^2]$, obtained by non-homogeneous dilation of the unit cube $[0,1]^3$. These represent a fundamental domain for the quotient $\mathbb{H}_3 / \varepsilon \Gamma$ of the 3D Heisenberg group $\mathbb{H}_3$ by the (dilation of the) integer Heisenberg subgroup $\Gamma$ (a lattice). This is the basic example of \emph{nilmanifold}, (we thank R. Montgomery for pointing out this example).
For these fundamental domains, the first Dirichlet eigenvalue is unknown. However, we mention that for any Carnot group the reduction technique developed here can be further improved leading to a $\lambda_1$ estimate for cubes, via the technique sketched in Appendix~\ref{sec:closed-geodesics}.)

\subsection{Structure of the paper}

In Section~\ref{s:srg} we recall some basic definitions about sub-Riemannian geometry and sub-Laplacians. In Section~\ref{s:prelconstr} we introduce some preliminary constructions concerning integration on vector bundles that we need for the reduction procedure. In Section~\ref{sec:sets} we prove the main result of the paper, namely the reduced Santal\'o formula. Section~\ref{s:examples} is devoted to examples, and contains the general class of structures where our construction can be carried out. Finally, in Section~\ref{s:applications} we apply the reduced Santal\'o formula to prove Poincar\'e, Hardy, and isoperimetric-type inequalities.

\section{Sub-Riemannian geometry}\label{s:srg}

We give here only the essential ingredients for our analysis; for more details see \cite{nostrolibro,montgomerybook,rifford2014sub}. A sub-Riemannian manifold is a triple $(M,\distr,g)$, where $M$ is a smooth, connected manifold of dimension $n \geq 3$, $\distr$ is a vector distribution of constant rank $k \leq n$ and $g$ is a smooth metric on $\distr$. We assume that the distribution is bracket-generating, that is 
\begin{equation}\label{eq:bracket-gen}
\spn\{[X_{i_1},[X_{i_2},[\ldots,[X_{i_{m-1}},X_{i_m}]]]]\mid m \geq 1\}_q = T_q M, \qquad \forall q \in M,
\end{equation}
for some (and thus any) set $X_1,\ldots,X_k \in \Gamma(\distr)$ of local generators for $\distr$. 

A \emph{horizontal curve} $\gamma : [0,T] \to \R$ is a Lipschitz continuous path such that $\dot\gamma(t) \in \distr_{\gamma(t)}$ for almost any $t$. Horizontal curves have a well defined length
\begin{equation}
\ell(\gamma) = \int_0^T \sqrt{g(\dot\gamma(t),\dot\gamma(t))}dt.
\end{equation}
Furthermore, the \emph{sub-Riemann\-ian distance} is defined by:
\begin{equation}
\d(x,y) = \inf\{\ell(\gamma)\mid \gamma(0) = x,\, \gamma(T) = y,\, \gamma \text{ horizontal} \}.
\end{equation}
By the Chow-Rashevskii theorem, under the bracket-generating condition, $\d$ is finite and continuous. Sub-Riemann\-ian geometry includes the Riemannian one, when $\distr = TM$.

\subsection{Sub-Riemannian geodesic flow}

\emph{Sub-Riemann\-ian geodesics} are horizontal curves that locally minimize the length between their endpoints. Let $\pi : T^*M \to M$ be the cotangent bundle. The \emph{sub-Riemann\-ian Hamiltonian} $H : T^*M \to \R$ is
\begin{equation}
H(\lambda) := \frac{1}{2}\sum_{i=1}^k \langle \lambda, X_i \rangle^2, \qquad \lambda \in T^*M,
\end{equation}
where $X_1,\ldots,X_k \in \Gamma(\distr)$ is any local orthonormal frame and $\langle \lambda, \cdot \rangle $ denotes the action of covectors on vectors. Let $\sigma$ be the canonical symplectic $2$-form on $T^*M$. The \emph{Hamiltonian vector field} $\vec{H}$ is defined by $\sigma(\cdot, \vec{H}) = dH$. Then the Hamilton equations are
\begin{equation}\label{eq:Hamiltoneqs}
\dot{\lambda}(t) = \vec{H}(\lambda(t)).
\end{equation}
Solutions of~\eqref{eq:Hamiltoneqs} are called \emph{extremals}, and their projections $\gamma(t) := \pi(\lambda(t))$ on $M$ are smooth geodesics. The \emph{sub-Riemann\-ian geodesic flow} $\phi_t \in T^*M \to T^*M$ is the flow of $\vec{H}$. Thus, any initial covector $\lambda \in T^*M$ is associated with a geodesic $\gamma_\lambda(t) = \pi\circ\phi_t(\lambda)$, and its speed $\|\dot\gamma(t)\|= 2H(\lambda)$ is constant. The \emph{unit cotangent bundle} is
\begin{equation}
U^*M  =\{\lambda  \in T^*M \mid  2H(\lambda) = 1 \}.
\end{equation}
It is a fiber bundle with fiber $U_q^* M = \mathbb{S}^{k-1}\times \R^{n-k}$. For $\lambda \in U_q^* M$, the curve $\gamma_\lambda(t)$ is a \emph{length-parametrized} geodesic with length $\ell(\gamma|_{[t_1,t_2]}) = t_2-t_1$.

\begin{rmk}
There is also another class of minimizing curves, called \emph{abnormal}, that might not follow the Hamiltonian dynamic of~\eqref{eq:Hamiltoneqs}. Abnormal geodesics do not exist in Riemannian geometry, and they are all trivial curves in some basic but popular classes of sub-Riemann\-ian structures (e.g.\ fat ones). Our construction takes in account only the normal sub-Riemann\-ian geodesic flow, hence abnormal geodesics are allowed, but ignored. Some hard open problems in sub-Riemann\-ian geometry are related to abnormal geodesics \cite{AAA-openproblems,montgomerybook,Sard-prop}.
\end{rmk}

\subsection{The intrinsic sub-Laplacian}

Let $(M,\distr,g)$ be a compact sub-Riemann\-ian manifold with piecewise $C^{1,1}$ boundary $\partial M$, and $\omega \in \Lambda^n M$ be any smooth volume form (or a density, if $M$ is not orientable). We define the \emph{Dirichlet energy functional} as
\begin{equation}
E(f) = \int_M 2 H(d f)\, \omega, \qquad f \in C^\infty_0(M).
\end{equation}
The Dirichlet energy functional induces the operator $-\Delta_{\omega}$ on $L^2(M,\omega)$.
Its Friedrichs extension is a non-negative self-adjoint operator on $L^2(M,\omega)$ that we call the \emph{Dirichlet sub-Laplacian}.
Its domain is the space $H^1_0(M)$, the closure in the $H^1(M)$ norm of the space $C^\infty_0(M)$ of smooth functions that vanish on $\partial M$. Since $\|\nabla_H f \|^2 = 2H( df)$, for smooth functions we have
\begin{equation}
\Delta_\omega f = \dive_\omega( \nabla_H f), \qquad \forall f \in C^\infty_0(M),
\end{equation}
where the divergence is computed w.r.t.\ $\omega$, and $\nabla_H$ is the horizontal gradient defined by \eqref{eq:horgrad}. The spectrum of $-\Delta_\omega$ is discrete and positive,
\begin{equation}
0 < \lambda_1(M) \le \lambda_2(M) \le \ldots \to +\infty .
\end{equation}
In particular, by the min-max principle we have
\begin{equation}\label{eq:minmax}
	\lambda_1(M) = \inf \left\lbrace E(f) \:\bigg|\: f\in C^\infty_0(M), \quad \int_M |f|^2\, \omega = 1 \right\rbrace.
\end{equation}


\section{Preliminary constructions}\label{s:prelconstr}

We discuss some preliminary constructions concerning integration on vector bundles that we need for the reduction procedure. In this section $\pi: E \to M$ is a rank $k$ vector bundle on an $n$ dimensional manifold $M$. For simplicity we assume $M$ to be oriented and $E$ to be oriented (as a vector bundle). If not, the results below remain true replacing volumes with densities. We use coordinates $x$ on $O \subset M$ and $(p,x) \in \R^k \times \R^n$ on $U=\pi^{-1}(O)$ such that the fibers are $E_{q_0} = \{(p,x_0)\mid p \in \R^k\}$. In a compact notation we write, in coordinates, $dp = dp_1\wedge \ldots \wedge dp_k$ and $dx = dx_1\wedge \ldots \wedge dx_n$.

\subsection{Vertical volume forms}\label{sec:vertical_forms}

Consider the fibers $E_q \subset E$ as embedded submanifolds of dimension $k$. For each $\lambda \in E_q$, let $\Lambda^k (T_\lambda E_{q})$ be the space of alternating multi-linear functions on $T_\lambda E_{q}$. The space
\begin{equation}
\Lambda^k_{\mathrm{v}} (E):= \bigsqcup_{\lambda \in E} \Lambda^k( T_\lambda E_{\pi(\lambda)} )
\end{equation}
defines a rank $1$ vector bundle $\Pi: \Lambda^k_{\mathrm{v}} (E) \to E$, such that $\Pi(\eta) = \lambda$ if $\eta \in \Lambda^k( T_\lambda E_{\pi(\lambda)} )$.

To see this, choose coordinates $(p,x) \in \R^k \times \R^n$ on $U=\pi^{-1}(O)$ such that the fibers are $E_{q_0} = \{(p,x_0)\mid p \in \R^k\}$. Thus the vectors $\partial_{p_1},\ldots,\partial_{p_k}$ tangent to the fibers $E_q$ are well defined. The map $\Psi: \Pi^{-1}(U) \to U \times \R$, defined by $\Psi(\eta)= (\Pi(\eta), \eta(\partial_{p_1},\ldots,\partial_{p_k}))$ is a bijection. Suppose that $(U',p',x')$ is another chart, and similarly $\Psi': \Pi^{-1}(U') \to U' \times \R$. Then, on $\Pi^{-1}(U'\cap U) \times \R$ we have $\Psi' \circ \Psi^{-1}(\lambda, \alpha) = (\lambda,\det (\partial q' /\partial q))$. Finally, we apply the vector bundle construction Lemma \cite[Lemma 5.5]{LeeDG}.
	
\begin{definition}
A smooth, strictly positive section $\nu \in \Gamma(\Lambda^k_{\mathrm{v}} (E))$ is called a \emph{vertical volume form on $E$}. In particular, the restriction $\nu_q:= \nu|_{E_q}$ of a vertical volume form defines a measure on each fiber $E_q$. 
\end{definition}
\begin{lemma}[Disintegration 1]
\label{l:vert-split}
Fix a volume form $\Omega \in \Lambda^{n+k}(E)$ and a volume form $\omega \in \Lambda^n(M)$ on the base space. Then there exists a unique vertical volume form $\nu \in \Lambda^{k}_\mathrm{v}(E)$ such that, for any measurable set $D' \subseteq E$ and measurable $f:D \to \R$,
\begin{equation}\label{eq:dis1}
\int_{D'} f\, \Omega = \int_{\pi(D')} \left[\int_{D_q'} f_q\, \nu_q \right] \omega(q), \qquad f_q:=f|_{E_q}, \quad D_q':= E_q \cap D'.
\end{equation}
If, in coordinates, $\Omega = \Omega(p,x) dp \wedge dx$ and $\omega = \omega(x) dx$, then
\begin{equation}
\nu|_{(p,x)} = \frac{\Omega(p,x)}{\omega(x)} dp.
\end{equation}
\end{lemma}
\begin{proof}
The last formula does not depend on the choice of coordinates $(p,x)$ on $E$. So we can use this as a definition for $\nu$. 
Moreover, in coordinates,
\begin{equation}
\Omega|_{(p,x)} = \Omega(p,x) dp \wedge dx = \left(\frac{\Omega(p,x)}{\omega(x)} dp\right) \wedge (\omega(x) dx).
\end{equation}
Both uniqueness and~\eqref{eq:dis1} follow from the definition of integration on manifolds and Fubini theorem.
\end{proof}

\subsection{Vertical surface forms} \label{sec:vert-surface}

Let $E'\subset E$ be a corank $1$ sub-bundle of $\pi:E \to M$. That is, a submanifold $E' \subset E$ such that $\pi|_{E'}: E' \to M$ is a bundle, and the fibers $E'_q:= \pi^{-1}(q)\cap E' \subset E_q$ are diffeomorphic to a smooth hypersurface $C \subset \R^k$. As a matter of fact, we will only consider the cases in which $C$ is a cylinder or a sphere.

Fix a smooth volume form $\Omega \in \Lambda^{n+k}(E)$. The \emph{Euler vector field} is the generator of homogenous dilations on the fibers $\lambda \mapsto e^\alpha \lambda$, for all $\alpha \in \R$. In coordinates $(p,x)$ on $E$ we have $\euler = \sum_{i=1}^n p_i \partial_{p_i}$. If $\euler$ is transverse to $E'$ we induce a volume form on $E'$ by $\mu:=\iota_\euler \Omega$.

In this setting, a volume form $\mu \in \Lambda^{n+k-1}(E')$ is called a \emph{surface form}.
For any vertical volume form $\nu\in\Lambda^k_{\mathrm{v}}(E)$, we define a measure on the fibers $E'_q$ as $\eta_q=\iota_\euler \nu|_{E_q}$.
 With an abuse of language, we will refer to such measures as \emph{vertical surface forms}.

\begin{lemma}[Disintegration 2]
\label{l:vert-split-sub}
Fix a surface form $\mu = \iota_\euler \Omega \in \Lambda^{n+k-1}(E')$ and a volume form $\omega \in \Lambda^n(M)$ on the base space. 
For any measurable set $D \subseteq E'$ and measurable $f:D \to \R$,
\begin{equation}\label{eq:dis2}
\int_D f \mu = \int_{\pi(D)} \left[\int_{D_q} f_q\, \eta_q \right] \omega(q), \qquad f_q:=f|_{E'_q}, \quad D_q:= E'_q \cap D.
\end{equation}
Here, $\eta_q = \iota_{\euler} \nu|_{E_q}$ and $\nu$ is the vertical volume form on $E$ defined in Lemma~\ref{l:vert-split}.
\end{lemma}
\begin{proof}
Choose coordinates $(p,x)$ on $E$. As in the proof of Lemma~\ref{l:vert-split}
\begin{align}
 \mu|_{(p,x)} = \iota_{\euler} \Omega|_{(p,x)} & = \Omega(p,x) \left(\iota_{\euler} dp \wedge dx +(-1)^k dp \wedge \iota_{\euler} dx \right) \\
 & = \Omega(p,x) \iota_{\euler} dp \wedge dx \\
& = \left(\frac{\Omega(p,x)}{\omega(x)} \iota_\euler dp \right) \wedge(\omega(x) dx).
\end{align}
Thus \eqref{eq:dis2} holds with $\eta|_{(p,x)} = \frac{\Omega(p,x)}{\omega(x)} \iota_\euler dp $. This, together with the local expression of $\nu$ in Lemma~\ref{l:vert-split}, yields $\eta = \iota_\euler \nu$.
\end{proof}

\begin{example}[The unit cotangent bundle]\label{ex:ucb}
We apply the above constructions to $E = T^*M$ and $E' = U^*M$. In this case $E'_q = U_q^* M$ are diffeomorphic to cylinders (or spheres, in the Riemannian case). Moreover, we set $\Omega = \Theta$, the Liouville volume form, and $\mu = \iota_\euler \Theta$, the Liouville surface form.\footnote{ Let $\theta\in \Lambda^1 (T^*M)$ be the \emph{tautological form} $\theta(\lambda) := \pi^*(\lambda)$.
The \emph{Liouville invariant volume} $\Theta\in \Lambda^{2n} (T^*M)$ is $\Theta := (-1)^{\frac{n(n-1)}2} d\theta\wedge\ldots\wedge d\theta$.
In canonical coordinates $(p,x)$ on $T^*M$ we have $\Theta= dp \wedge dx $.} One can check that $\Theta = d\mu$.

Let $\nu \in \Lambda_{\mathrm{v}}^n (T^*M)$ and $\eta = \iota_\mathfrak{e}\nu$ as in Lemmas~\ref{l:vert-split} and \ref{l:vert-split-sub}. In canonical coordinates, $\Theta = dp\wedge dx$. Then, if $\omega = \omega(x) dx$,
\begin{equation}
	\nu = \frac 1 {\omega(x)} \, dp \quad \text{ and }\quad \eta = \frac 1 {\omega(x)} \sum_{i=1}^n (-1)^{i-1} p_i\,dp_1\wedge\ldots\wedge \widehat{dp_i}\wedge\ldots\wedge dp_n.
\end{equation}
Choose coordinates $x$ around $q_0 \in M$ such that $\partial_{x_1}|_{q_0},\ldots,\partial_{x_k}|_{q_0}$ is an orthonormal basis for the sub-Riemannian distribution $\distr_{q_0}$. In the associated canonical coordinates we have
	\begin{equation}
		U_{q_0}^* M = \{ (p,x_0) \in\R^{2n}\mid p_1^2+\ldots+p_k^2 = 1 \} \simeq \mathbb S^{k-1}\times \R^{n-k}.
	\end{equation}
In this chart, $\eta_{q_0}$ is the $(n-1)$-volume form of the above cylinder times $1/\omega(x_0)$.
\end{example}

\begin{rmk}
This construction gives a canonical way to define a measure on $U^*M$ and its fibers in the general sub-Riemannian case, depending only on the choice of the volume $\omega$ on the manifold $M$. It turns out that this measure is also invariant under the Hamiltonian flow. Notice though that in the sub-Riemannian setting, fibers have infinite volume. 
\end{rmk}

\subsection{Invariance}\label{sec:invariance}

Here we focus on the case of interest where $E \subseteq T^*M$ is a rank $k$ vector sub-bundle and $E' \subset E$ is a corank $1$ sub-bundle as defined in Section~\ref{sec:vert-surface}. We stress that $E'$ is not necessarily a vector sub-bundle, but typically its fibers are cylinders or spheres.

Recall that the sub-Riemannian geodesic flow $\phi_t:T^*M \to T^*M$ is the Hamiltonian flow of $H: T^*M \to \R$. Moreover, in our picture, $M \subset N$ is a compact submanifold with boundary $\partial M$ of a larger manifold $N$, with $\dim M = \dim N = n$.

\begin{definition}\label{def:invariant}
A sub-bundle $E\subseteq T^*M$ is \emph{invariant} if $\phi_t(\lambda) \in E$ for all $\lambda \in E$ and $t$ such that $\phi_t(\lambda) \in T^*M$ is defined. A volume form $\Omega \in \Lambda^{n+k}(E)$ is \emph{invariant} if $\mathcal{L}_{\vec{H}} \Omega = 0$.
\end{definition}
Our definition includes the case of interest for Santal\'o formula, where sub-Riemannian geodesics may cross $\partial M \neq \emptyset$. In other words, $E$ is invariant if the only way to escape from $E$ through the Hamiltonian flow is by crossing the boundary $\pi^{-1}(\partial M)$. Moreover, if $\Omega$ is an invariant volume on an invariant sub-bundle $E$, then $\phi_t^* \Omega = \Omega$.

\begin{lemma}[Invariant induced measures]\label{l:invariance}
Let $E \subseteq T^*M$ be an invariant vector bundle with an invariant volume $\Omega$. Let $E' \subset E$ be a corank $1$ invariant sub-bundle. Let $\euler$ be a vector field transverse to $E'$ and $\mu = \iota_\euler \Omega$ the induced surface form on $E'$. Then $\mu$ is invariant if and only if $[\vec{H},\euler]$ is tangent to $E'$.
\end{lemma}
In Example~\ref{ex:ucb}, $E = T^*M$ and $E' = U^*M$ are clearly invariant; in particular $\vec{H}$ is tangent to $E'$. By Liouville theorem, $\Omega = \Theta$ is invariant for any Hamiltonian flow Moreover, if the Hamiltonian $H$ is homogeneous of degree $d$ (on fibers), one checks that $[\vec{H}, \euler] = - (d-1)\vec{H}$ and Lemma~\ref{l:invariance} yields the invariance of the Liouville surface measure $\mu = \iota_\euler \Theta$. In particular this holds in Riemannian and sub-Riemannian geometry, with $d=2$.

\section{Santal\'o formula} \label{sec:sets}

\subsection{Assumptions on the boundary}

Let $(N,\distr,g)$ be a smooth connected sub-Rie\-mann\-ian manifold, of dimension $n$, without boundary. We focus on a compact $n$-dimen\-sional submanifold $M$ with piecewise $C^{1,1}$ boundary $\partial M$.

Let $q \in \partial M$ such that the tangent space is well defined. We say that $q$ is a \emph{characteristic point} if $\distr_q \subseteq T_q \partial M$. If $q$ is \emph{non-characteristic}, the \emph{horizontal normal} at $q$ is the unique inward pointing unit vector $\mathbf{n}_q \in \distr_q$ orthogonal to $T_q \partial M \cap \distr_q$. If $q\in\partial M$ is characteristic, we set $\mathbf{n}_q=0$.  We call $C(\partial M)$ the set of characteristic points. The size of  $C(\partial M)$ has been studied in \cite{derridj,balogh2} under various regularity assumptions on $\partial M$. We give a self-contained proof of the negligibilty of $C(\partial M)$, which we need in the following. The $C^{1,1}$ regularity assumption cannot be weakened to $C^{1,\alpha}$, with $0<\alpha<1$, as shown in \cite[Thm. 1.4]{balogh-heis}.
\begin{prop}\label{p:h0}
	Let $\partial M$ be piecewise $C^{1,1}$. Then, the set of characteristic points $C(\partial M)$ has zero measure in $\partial M$.
\end{prop}
\begin{proof}
Without loss of generality we assume that $N=\mathbb R^n$ and that locally $\partial M$ is the graph of a $C^{1,1}$ function $f: \R^{n-1} \to \R$. Let also $u(x,z) = z-f(x)$, so that locally $\partial M = \{(x,z) \in \R^{n-1}\times \R \mid u (x,z)=0 \}$.

Let $\tilde{A} \subset \R^{n-1}$ be measurable with positive measure, and let $A = \{(x,f(x)) \mid x \in \tilde{A} \}\subseteq \partial M$. We claim that if $X,Y$ are smooth vector fields (not necessarily horizontal), tangent to $\partial M$ a.e.\ on $A$, then also $[X,Y]$ is tangent to $\partial M$ a.e.\ on $A$. Notice that $X$ is tangent to $\partial M$ a.e.\ on $A$ if and only if $X(u)(x,f(x)) = 0$ for a.e.\ $x \in \tilde{A}$. Consider the Lipschitz function $\xi(x):= X(u)(x,f(x))$. As a consequence of coarea formula \cite{MR3409135}, we have
\begin{equation}
\int_{\tilde{A}} |\nabla \xi(x)|\, dx = \int_{\R} \mathcal{H}^{n-2}(\tilde{A} \cap \xi^{-1}(t) )\, dt = 0,
\end{equation}
where $|\nabla \xi|$ is the norm of the Euclidean gradient of $\xi : \R^{n-1} \to \R$, and $\mathcal{H}^{n-2}$ is the Hausdorff measure. In particular, $\nabla \xi = 0$ a.e.\ on $\tilde{A}$. Since $Y$ is tangent to $\partial M$ a.e.\ on $A$, the above identity yields that $Y(X(u)) =0$ a.e.\ on $A$. A similar argument shows that also $Y(X(u)) =0$ a.e.\ on $A$. Since $[X,Y](u)(x,f(x)) = X(Y(u))(x,f(x))-Y(X(u))(x,f(x))$ for a.e.\ $x \in \R^{n-1}$, we have that $[X,Y]$ is tangent to $\partial M$ a.e.\ on $A$, as claimed. 

Assume by contradiction that $C(\partial M)$ has positive measure. In particular, applying the above claim to any pair $X,Y \in \Gamma(\distr)$, and $A = C(\partial M)$, we obtain that $[X,Y]$ is tangent to $\partial M$ a.e. on $C(\partial M)$. Since $[X,Y] \in \Gamma(TM)$, we can apply the claim a finite number of times, obtaining that any iterated Lie bracket of elements of $\Gamma(\distr)$ is tangent to $\partial M$ a.e. on $C(\partial M)$. This contradicts the bracket-generating assumption.
\end{proof}

\subsection{(Sub-)Riemannian Santal\'o formula}

For any covector $\lambda \in U^*_q M$, the \emph{exit length} $\ell(\lambda)$ is the first time $t \geq 0$ at which the corresponding geodesic $\gamma_\lambda(t)= \pi\circ\phi_t(\lambda)$ leaves $M$ crossing its boundary, while $\tilde{\ell}(\lambda)$ is the smallest between the exit and the cut length along $\gamma_\lambda(t)$. Namely
\begin{align}
\ell(\lambda) & = \sup\{ t\ge 0\mid \gamma_\lambda(t)\in M \}, \\
\tilde \ell(\lambda) & = \sup \{t\le \ell(\lambda)\mid \gamma_\lambda|_{[0,t]} \text{ is minimizing} \}.
\end{align}
We also introduce the following subsets of the unit cotangent bundle $\pi: U^*M \to M$:
\begin{align}
{U^+\partial M} & = \left\{ \lambda\in U^*M|_{\partial M}  \mid \langle \lambda, \textbf{n}\rangle > 0  \right\}, \\
	 U^\vis M & = \{ \lambda\in U^*M \mid \ell(-\lambda) <+\infty\} , \\
	 \tilde U^\vis M & = \{ \lambda\in U^\vis M \mid \tilde\ell(-\lambda) = \ell(-\lambda)\}.
\end{align}
Some comments are in order. The set $U^+\partial M$ consists of the unit covectors $\lambda \in \pi^{-1}(\partial M)$ such that the associated geodesic enters the set $M$ for arbitrary small $t>0$. The \emph{visible set} $U^\vis M$ is the set of covectors that can be reached in finite time starting from $\pi^{-1}(\partial M)$ and following the geodesic flow. If we restrict to covectors that can be reached \emph{optimally} in finite time, we obtain the \emph{optimally visible set} $\tilde U^\vis M$ (see Fig.~\ref{f:fig1}).

\begin{lemma}\label{l:cuttime}
	The cut-length $c:U^*M\to (0,+\infty]$ is upper semicontinuous (and hence measurable).
	Moreover, if any couple of distinct points in $M$ can be joined by a minimizing non-abnormal geodesic, $c$ is continuous.
\end{lemma}

\begin{proof}
	The result follows as in \cite[Thm. III.2.1]{chavelbook}.
	We stress that the key part of the proof of the second statement is the fact that, in absence of non-trivial abnormal minimizers, a point is in the cut locus of another if and only if (i) it is conjugate along some minimizing geodesic or (ii) there exist two distinct minimizing geodesics joining them.
\end{proof}

\begin{lemma}\label{l:semicont}
	The exit length $\ell:U^+\partial M \to (0, +\infty]$ is lower semicontinuous (and hence measurable).
	Moreover, $\tilde\ell:U^+\partial M \to (0, +\infty]$ is measurable.	
\end{lemma}

\begin{proof}
	Let $\lambda_0\in U^+\partial M$. Consider a sequence $\lambda_n$ such that $\liminf_{\lambda\rightarrow\lambda_0} \ell(\lambda) = \lim_{n} \ell(\lambda_n)$.
	Then, the trajectories $\gamma_n(t) = \pi\circ \phi_t(\lambda_n)$ for $t\in [0,\ell(\lambda_n)]$ converge uniformly as $n\rightarrow+\infty$ to the trajectory $\gamma_0(t)=\phi_t(\lambda_0)$ for $t\in [0,\delta]$ where $\delta = \lim_n\ell(\lambda_n)$.
	Moreover, by continuity of $\partial M$ and the fact that $\gamma_n(\ell(\lambda_n))\in \partial M$, it follows that $\gamma_0(\delta)\in \partial M$.
	This proves that $\delta\ge \ell(\lambda_0)$, proving the first part of the statement.

	To complete the proof, observe that $\tilde \ell = \min\{\ell, c\}$, which are measurable by the previous claim and Lemma~\ref{l:cuttime}.
\end{proof}

Fix a volume form $\omega$ on $M$ (or density, if $M$ is not orientable). In any case, $\omega$ and $\sigma:=\iota_{\mathbf{n}} \omega$ induce positive measures on $M$ and $\partial M$, respectively.
According to Lemmas \ref{l:vert-split} and \ref{l:vert-split-sub}, these induce measures $\nu_q$ and $\eta_q = \iota_\euler \nu_q$ on $T_q^*M$ and $U_q^*M$, respectively.
\begin{theorem}[Santal\'o formulas]\label{t:santalo}
The visible set $U^\vis M$ and the optimally visible set $\tilde{U}^\vis M$ are measurable. Moreover, for any measurable function $F: U^*M \to \R$ we have
\begin{align}\label{eq:santalo}
\int_{U^\vis M} F \,\liousurf
		& = \int_{\partial M} \left[\int_{U^+_q\partial M}  \left( \int_0^{\ell(\lambda)}  F ( \phi_t(\lambda))  dt\right)  \langle \lambda,{\mathbf{n}_q}\rangle \eta_q(\lambda)\right] \sigma(q), \\
	\label{eq:santalo-tilde} 
\int_{\tilde{U}^\vis M} F \,\liousurf
		& = \int_{\partial M} \left[\int_{U^+_q\partial M}  \left( \int_0^{\tilde{\ell}(\lambda)}  F ( \phi_t(\lambda))  dt\right)  \langle \lambda,{\mathbf{n}_q}\rangle \eta_q(\lambda)\right] \sigma(q).
\end{align}
\end{theorem}
\begin{rmk}
	Even if $M$ is compact and hence $\tilde \ell<+\infty$, in general $\tilde U^\vis M\subsetneq U^\vis M$.
	Furthermore, if $\ell<+\infty$ (that is, all geodesics reach the boundary of $M$ in finite time), then $U^\vis M = U^*M$.
	Thus, our statement of Santal\`o formula contains \cite[Thm. VII.4.1]{chavelbook}.
\end{rmk}

\begin{rmk}\label{rmk:noH0}
If $\partial M$ is only Lipschitz and $C(\partial M)$ has positive measure, the above Santal\'o formulas still hold by removing on the left hand side from $U^\vis M$ and $\tilde U^\vis M$ the set $\{\phi_t(\lambda)\mid \pi(\lambda)\in C(\partial M) \text{ and } t\ge 0\}$. Nothing changes on the right hand side as $\sigma(C(\partial M))=0$, since $\sigma = \iota_{\mathbf{n}} \omega$ and $\mathbf{n}$ vanishes on $C(\partial M)$ by definition.
\end{rmk}

\begin{proof}
	Let $A \subset [0,+\infty) \times U^+  \partial M $ be the set of pairs $(t,\lambda)$ such that $0<t< \ell(\lambda)$. 
	By Lemma~\ref{l:semicont} it follows that $A$ is measurable.
	Let also $Z = \pi^{-1}(\partial M)\subset U^\vis M$ which has zero measure in $U^*M$.
	Define $\phi:A \to U^\vis M\setminus Z$ as $\phi(t,\lambda) = \phi_t(\lambda)$.
	This is a smooth diffeomorphism, whose inverse is $\phi^{-1}(\bar \lambda) = (\ell(-\bar\lambda),-\phi_{\ell(-\bar\lambda)}(-\bar\lambda))$. In particular, $U^\vis M$ is measurable.
	Then, using Lemma~\ref{l:splitting} (see below), and Fubini theorem, we have
	\begin{multline}
		\label{eq:santalo-deriv}
		\int_{U^\vis M} F\, \mu 
		= \int_{\phi(A)} F\, \mu = \int_{A} (F\circ \phi)\, \phi^*\mu =\\
		= \int_{\partial M} \left[\int_{U^+_q\partial M}  \left( \int_0^{\ell(\lambda)}  F ( \phi_t(\lambda))  dt\right)  \langle \lambda,{\mathbf{n}_q}\rangle \eta_q(\lambda)\right] \sigma(q),
	\end{multline}
	which proves \eqref{eq:santalo}. Analogously, with $\tilde A = \{ (t,\lambda) \mid 0<t< \tilde\ell(\lambda) \}$ and $\tilde Z = Z \cup \{ \phi_{\tilde\ell(\lambda)} (\lambda) \mid \lambda\in U^+\partial M \}$ the map $\phi: \tilde A \to \tilde U^\vis M\setminus \tilde Z$ is a diffeomorphism with the same inverse.
	Then, the same computations as \eqref{eq:santalo-deriv} replacing $A$ with $\tilde A$ and $Z$ with $\tilde Z$ yield \eqref{eq:santalo-tilde}. \end{proof}

\begin{lemma}
	\label{l:splitting}
The following local identity of elements of $\Lambda^{2n-1}(\R\times U^+\partial M)$ holds
	\begin{equation}
		\phi^* \mu|_{(t,\lambda)} = \langle \lambda, \mathbf{n}_q \rangle\, dt\wedge \sigma \wedge \eta , \qquad \lambda \in U^+\partial M,
	\end{equation}
where, in canonical coordinates $(p,x)$ on $T^*M$
\begin{equation}
\eta = \iota_\euler \nu, \qquad \nu = \frac{1}{\omega(x)} dp, \qquad \sigma = \iota_{\mathbf{n}} \omega, \qquad \omega = \omega(x) dx.
\end{equation}

\end{lemma}

\begin{proof}
	For any $(t,\lambda)\in\R\times U^+\partial M$ let $\{\partial_t, v_1,\ldots, v_{2n-2}\}$ be a set of independent vectors in $T(\R\times U^+\partial M) = T\R \oplus TU^+\partial M$.
	Observe that $\phi^*\mu = dt\wedge (\iota_{\partial_t} \phi^*\mu)$.
	Then, 
	\begin{equation}
		\iota_{\partial_t} \phi^*\mu(v_1,\ldots,v_{2n-2}) 
		= \mu|_{\phi(t,\lambda)} \left(d_{(t,\lambda)}\phi \,\partial_t,d_{(t,\lambda)}\phi \,v_1,\ldots, d_{(t,\lambda)}\phi \,v_{2n-2}\right).
	\end{equation}
	Notice that,
	\begin{itemize}
		\item[(a)] $d_{(t,\lambda)}\phi \,\partial_t = (d_{\lambda}\phi_t) \, \vec H$, this is in fact just $\vec H|_{\phi_t(\lambda)}$,
		\item[(b)] $d_{(t,\lambda)}\phi \, v_i = (d_\lambda \phi_t) v_i$ for any $i=1,\ldots, 2n-2$.
	\end{itemize} 
	Hence it follows that 
	\begin{equation}
		\label{eq:splitting1}
		\iota_{\partial_t} \phi^*\mu = \iota_{\vec H} \phi_t^*\mu=\iota_{\vec H}\mu,
	\end{equation} 
	where in the last passage we used the invariance of $\mu$ (see the discussion below Lemma~\ref{l:invariance}). By Lemma~\ref{l:vert-split-sub} and its proof (in particular see Example~\ref{ex:ucb}) locally $\mu = \eta\wedge \omega$.	By the properties of the interior product, 
	\begin{equation}\label{eq:splitting2}
	\iota_{\vec{H}} \mu = (\iota_{\vec{H}} \eta) \wedge \omega + (\iota_{\vec{H}} \omega) \wedge \eta.
	\end{equation}
	The first term on the r.h.s.\ vanishes: as a $2n-2$ form, its value at a point $\lambda \in U^+\partial M$ is completely determined by its action on $2n-2$ independent vectors of $T_{\lambda} U^+\partial M$. We can choose coordinates such that $\partial M = \{x_n = 0\}$. Then a basis of $T_{\lambda} U^+\partial M$ is given by $\partial_{x_1}, \ldots,\partial_{x_{n-1}}$ and a set of $n-1$ vectors $v_{i} = \sum_{j=1}^n v_i^j \partial_{p_j}$ in $T_{\lambda} U_{\pi(\lambda)}^+\partial M$. Since $\iota_{\vec{H}}\eta$ is a $n-2$ form, then $\omega$ necessarily acts on at least one $v_i$, and vanishes. Now, notice that 
	\begin{equation}
		\label{eq:splitting3}
		\iota_{\vec H}  \omega |_{\lambda}(\cdot) = \omega|_{\pi(\lambda)}( \pi_*\vec H, \pi_*\cdot) 
		= \langle \lambda,\textbf{n}_{\pi(\lambda)} \rangle \, \omega|_{\pi(\lambda)}(\textbf{n}_{\pi(\lambda)}, \pi_*\cdot) 
		= \langle \lambda ,\textbf{n}_{\pi(\lambda)} \rangle \sigma |_{\lambda}(\cdot).
	\end{equation}
	Putting together \eqref{eq:splitting1}, \eqref{eq:splitting2}, and \eqref{eq:splitting3} completes the proof of the statement.
\end{proof}

\subsection{Reduced Santal\'o formula}\label{sec:redsant}

The following reduction procedure replaces the non-compact set $U^\vis M$ in Theorem~\ref{t:santalo} with a compact subset that we now describe.

To carry out this procedure we fix a transverse sub-bundle $\ver\subset TM$ such that $TM=\distr\oplus \ver$. We assume that $\ver$ is the orthogonal complement of $\distr$ w.r.t. to a Riemannian metric $g$ such that $g|_{\distr}$ coincides with the sub-Riemannian one and the associated Riemannian volume coincides with $\omega$. In the Riemannian case, where $\ver$ is trivial, this forces $\omega = \omega_R$, the Riemannian volume. In the genuinely sub-Riemannian case there is no loss of generality since this assumption is satisfied for any choice of $\omega$.

\begin{definition}\label{def:redcotan}
The \emph{reduced cotangent bundle} is the rank $k$ vector bundle $\pi: T^*M^\mathsf{r} \to M$ of covectors that annihilate the vertical directions:
\begin{equation}
T^*M^\mathsf{r} := \left\lbrace \lambda\in T^*M\mid \langle \lambda, v\rangle = 0 \text{ for all } v\in \ver \right\rbrace.
\end{equation}
The \emph{reduced unit cotangent bundle} is $U^*M^\mathsf{r} := U^*M\cap T^*M^\mathsf{r}$.
\end{definition}
Observe that $U^*M^\mathsf{r}$ is a corank $1$ sub-bundle of $T^*M^\mathsf{r}$, whose fibers are spheres $\mathbb S^{k-1}$. If $T^*M^\mathsf{r}$ is invariant in the sense of Definition~\ref{def:invariant}, we can apply the construction of Section~\ref{sec:invariance}. The Liouville volume $\Theta$ on $T^*M$ induces a volume on $T^*M^\mathsf{r}$ as follows.

Let $X_1,\ldots,X_k$ and $Z_1,\ldots,Z_{n-k}$ be local orthonormal frames for $\distr$ and $\ver$, respectively. Let $u_i(\lambda):=\langle \lambda,X_i)$ and $v_j(\lambda):=\langle \lambda, Z_j\rangle$ smooth functions on $T^*M$. Thus
\begin{equation}
T^*M^\mathsf{r} = \{ \lambda \in T^*M \mid v_1(\lambda) = \ldots = v_{n-k}(\lambda) = 0 \}.
\end{equation}
For all $q \in M$ where the fields are defined, $(u,v): T^*_qM\to \R^n$ are smooth coordinates on the fiber and hence $\partial_{u_1}, \ldots, \partial_{u_k}$,$\partial_{v_1},\ldots,\partial_{v_{n-k}}$ are vectors on $T_\lambda(T_q^*M) \subset T_\lambda(T^*M)$ for all $\lambda\in\pi^{-1}(q)$. In particular, the vector fields $\partial_{v_1},\ldots,\partial_{v_{n-k}}$ are transverse to $T^*M^\mathsf{r}$, hence we give the following definition.

\begin{definition}\label{def:redLiouville}
The \emph{reduced Liouville volume} $\Theta^\mathsf{r} \in \Lambda^{n+k}(T^*M^\mathsf{r})$ is
\begin{equation}
\Theta^\mathsf{r}_\lambda := \Theta_\lambda( \underbrace{\ldots,\ldots,\ldots}_{k \text{ vectors}}, \partial_{v_1},\ldots,\partial_{v_{n-k}},\underbrace{\ldots,\ldots,\ldots}_{n \text{ vectors}} ), \qquad \forall\lambda\in T^*M^\mathsf{r}.
\end{equation}
\end{definition}
The above definition of $\Theta^\mathsf{r}$ does not depend on the choice of the local orthonormal frame $\{X_1,\ldots, X_k,Z_1,\ldots,Z_{n-k}\}$ and Riemannian metric $g|_\ver$ on the complement, as long as its Riemannian volume remains the fixed one, $\omega$.
In fact, let $X',Z'$ be a different frame for a different Riemannian metric $g'|_\ver$.
Then\footnote{For simplicity, assume that $\distr$ is orientable as a vector bundle and that $X_1\ldots, X_k$ is an oriented frame.}, $X' = R X$ and $Z' = SX+TZ$ for $R\in \mathrm{SO}(k)$, $T\in \mathrm{SL}(n-k)$ and $S\in \mathrm{M}(k, n)$. One can check that $\partial_{v} = S\partial_{u'}+T\partial_{v'}$ and that
$	\Theta^\mathsf{r} 
	= \Theta(\ldots,\partial_v,\ldots) 
	= \Theta(\ldots,T\partial_{v'},\ldots) 
	= (\Theta^\mathsf{r} )'$,
 where both frames are defined.

\subsection*{Assumptions for reduction} We assume the following hypotheses:
\begin{itemize}
\item[(\textbf{H1})] The bundle $T^*M^\mathsf{r} \subseteq T^*M$ is invariant.
\item[(\textbf{H2})] The \emph{reduced Liouville volume} is invariant, i.e. $\mathcal{L}_{\vec{H}} \Theta^\mathsf{r} = 0$.
\end{itemize}

\begin{rmk}
Assumption (\textbf{H1}) depends only on $\ver$, while (\textbf{H2}) depends also on $\omega$ (since $\Theta^\mathsf{r}$ does). In the Riemannian case, with $\omega = \omega_R$, both are trivially satisfied.
\end{rmk}

Under these assumptions $U^*M^\mathsf{r} = U^*M \cap T^*M^\mathsf{r}$ is an invariant corank $1$ sub-bundle of $T^*M^\mathsf{r}$. Moreover, $\mu^\mathsf{r}= \iota_\euler \Theta^\mathsf{r}$ is an invariant surface form on $U^*M^\mathsf{r}$. This follows from Lemma~\ref{l:invariance} observing that $[\vec H, \euler] = -\vec H$ is tangent to $U^*M^\mathsf{r}$. As in Section~\ref{sec:vertical_forms}, the volume $\Theta^\mathsf{r}\in\Lambda^{n+k}(T^*M^\mathsf{r})$ induces a vertical volume $\nu^\mathsf{r}_q$ on the fibers $T_q^*M^\mathsf{r}$ and a vertical surface form $\eta^\mathsf{r}_q = \iota_\euler \nu_q$ on $U_q^*M^\mathsf{r}$.
As a consequence of Lemmas~\ref{l:vert-split} and \ref{l:vert-split-sub} the latter has the following explicit expression, whose proof is straightforward.
\begin{lemma}[Explicit reduced vertical measure]
	\label{l:eta-r}
	Let $q_0\in M$ and fix a set of canonical coordinates $(p,x)$ such that $q_0$ has coordinates $x_0$ and
	\begin{itemize}
		\item $\{\partial_{x_1}, \dots, \partial_{x_k}\}_{q_0}$ is an orthonormal basis of $\distr_{q_0}$,
		\item $\{\partial_{x_{k+1}}, \dots, \partial_{x_{n}}\}_{q_0}$ is an orthonormal basis of $\ver_{q_0}$.
	\end{itemize}
In these coordinates $\omega|_{x_0} = dx|_{q_0}$. Then $\nu^\mathbf{r}_{q_0} =\mathrm{vol}_{\R^k}$ and $\eta^\mathbf{r}_{q_0} = \mathrm{vol}_{\mathbb S^{k-1}}$.
	In particular, 
	\begin{equation}
		\int_{U_{q_0}^*M^\mathsf{r}} \eta^\mathsf{r}_{q_0} = |\mathbb S^{k-1}|, \qquad \forall q_0 \in M,
	\end{equation}
    where $|\mathbb{S}^{k-1}|$ denotes the Lebesgue measure of $\mathbb{S}^{k-1}$ and $\mathrm{vol}_{\R^k}$, $\mathrm{vol}_{\mathbb{S}^{k-1}}$ denote the Euclidean volume forms of $\R^k$ and $\mathbb{S}^{k-1}$.
\end{lemma}

We now state the reduced Santal\'o formulas.
The sets $U^+\partial M^\mathsf{r}$, $U^\vis M^\mathsf{r}$, and $\tilde U^\vis M^\mathsf{r}$ are defined from their unreduced counterparts by taking the intersection with $T^*M^\mathsf{r}$.

\begin{theorem}[Reduced Santal\'o formulas]
	\label{t:santalo-reduced}
	The visible set $U^\vis M^\mathsf{r}$ and the optimally visible set $\tilde{U}^\vis M^\mathsf{r}$ are measurable.
	For any measurable function $F:U^*M^\mathsf{r} \to \R$ we have
	\begin{align}
		\label{eq:santalo-reduced}
		\int_{U^\vis M^\mathsf{r}} F \,\liousurf^\mathsf{r}
			& = \int_{\partial M} \left[\int_{U_q^+\partial M^\mathsf{r}}  \left( \int_0^{\ell(\lambda)}  F ( \phi_t(\lambda))  dt\right)  \langle \lambda,{\mathbf{n}_q}\rangle \eta^\mathsf{r}_q(\lambda)\right] \sigma(q), \\
		\label{eq:santalo-tilde-reduced} 
		\int_{\tilde{U}^\vis M^\mathsf{r}} F \,\liousurf^\mathsf{r}
			& = \int_{\partial M} \left[\int_{U_q^+\partial M^\mathsf{r}}  \left( \int_0^{\tilde{\ell}(\lambda)}  F ( \phi_t(\lambda))  dt\right)  \langle \lambda,{\mathbf{n}_q}\rangle \eta^\mathsf{r}_q(\lambda)\right] \sigma(q).
\end{align}
\end{theorem}
\begin{proof}The proof follows the same steps as the one of Theorem~\ref{t:santalo} replacing the invariant sub-bundles, volumes, and surface forms with their reduced counterparts.
\end{proof}

\begin{rmk}\label{r:twoflows}
Let $H_{sR}$ be the sub-Riemannian Hamiltonian and $H_{R}$ be the Riemannian Hamiltonian of the Riemannian extension. The two Hamiltonians are (locally on $T^*M$)
\begin{equation}
H_R = \frac{1}{2}\left(\sum_{i=1}^k u_i^2 + \sum_{j=1}^{n-k}v_j^2\right), \qquad H_{sR} = \frac{1}{2} \sum_{i=1}^k u_i^2.
\end{equation}
Let $\phi_t^{sR} = e^{t \vec{H}_{sR}}$ and $\phi_t^R = e^{t\vec{H}_R}$ be their Hamiltonian flows. Since $T^*M^\mathsf{r} = \{\lambda\mid v_1(\lambda) =\ldots=v_{n-k}(\lambda) =0\}$, by assumption (\textbf{H1}) we have
\begin{equation}
H_{sR} = H_R, \quad \text{and}\quad \phi_t^{sR} = \phi_t^R \qquad \text{ on } T^*M^\mathsf{r}.
\end{equation}
In particular, the sub-Riemannian geodesics with initial covector $\lambda \in U^*M^\mathsf{r}$ are also geodesics of the Riemannian extension and viceversa.
\end{rmk}

\section{Examples}\label{s:examples}

\subsection{Carnot groups}\label{s:Carnot}

A \emph{Carnot group $(G,\star)$ of step $m$} is a connected, simply connected Lie group of dimension $n$, such that its Lie algebra $\mathfrak{g} = T_e G$ admits a nilpotent stratification of step $m$, that is
\begin{equation}
\mathfrak{g} = \mathfrak{g}_1 \oplus \ldots \oplus \mathfrak{g}_m,
\end{equation}
with
\begin{equation}
[\mathfrak{g}_1,\mathfrak{g}_j] = \mathfrak{g}_{1+j},\quad \forall 1\leq j\leq m, \quad \mathfrak{g}_m \neq \{0\}, \quad \mathfrak{g}_{m+1} =\{0\}.
\end{equation}
Let $\distr$ be the left-invariant distribution generated by $\mathfrak{g}_1$, and consider any left-invariant sub-Riemannian structure on $G$ induced by a scalar product on $\mathfrak{g}_1$. 

We identify $G \simeq \R^{n}$ with a polynomial product law by choosing a basis for $\mathfrak{g}$ as follows. Recall that the group exponential map,
\begin{equation}
\mathrm{exp}_{G} : \mathfrak{g} \to G,
\end{equation}
associates with $V \in \mathfrak{g}$ the element $\gamma(1)$, where $\gamma: [0,1] \to G$ is the unique integral line starting from $\gamma(0)=0$ of the left invariant vector field associated with $V$. Since $G$ is simply connected and $\mathfrak{g}$ is nilpotent, $\mathrm{exp}_G$ is a smooth diffeomorphism. 

Let $d_j:= \dim\mathfrak{g}_j$. Indeed $d_1 = k$. Let $\{X_{i}^{j}\}$, for $j=1,\ldots,m$ and $i=1,\ldots,d_j$ be an adapted basis, that is $\mathfrak{g}_j = \spn\{X_1^j,\ldots,X_{d_j}^j\}$. In exponential coordinates we identify
\begin{equation}
(x^1,\ldots,x^m) \simeq \mathrm{exp}_G\left( \sum_{j=1}^m \sum_{i=1}^{d_j}  x_i^j X_i^j\right), \qquad x^j \in \R^{d_j}.
\end{equation}
The identity $e \in G$ is the point $(0,\ldots,0) \in \R^n$ and, by the Baker-Cambpell-Hausdorff formula the group law $\star$ is a polynomial expression in the coordinates $(x^1,\ldots,x^m)$.  Finally,
\begin{equation}
X^j_i =\left.\frac{\partial}{\partial x^j_i}\right|_0,
\end{equation}
so that $\distr|_e \simeq \{(x,0,\ldots,0)\mid x \in \R^k\}$ and $\distr_q =  L_{q*} \distr|_e$, where $L_{q*}$ is the differential of the \emph{left-translation} $L_q (p) := q \star p$.

We equip $G$ with the Lebesgue volume of $\R^n$, which is a left-invariant Haar measure.
In order to apply the reduction procedure of Section~\ref{sec:redsant}, let $\mathcal{V}$ be the left-invariant distribution generated by
\begin{equation}
\mathcal{V}|_e:= \mathfrak{g}_2 \oplus \ldots \oplus \mathfrak{g}_m,
\end{equation}
and consider any left-invariant scalar product $g|_{\ver}$ on $\mathcal{V}$. Thus, up to a renormalization, $g = g|_{\distr} \oplus g|_{\ver}$ is a left-invariant Riemannian extension such that $TM = \distr \oplus \ver$ is an orthogonal direct sum and its Riemannian volume coincides with the Lebesgue one.

\begin{prop}\label{p:CarnotOK}
Any Carnot group satisfies assumptions (\textbf{H1}) and (\textbf{H2}).
\end{prop}
\begin{proof}
Let $X_1,\ldots,X_k  \in \Gamma(\distr)$ and $Z_1,\ldots,Z_{n-k} \in \Gamma(\ver)$ be a global frame of left-invariant orthonormal vector fields. Let $u_i(\lambda):= \langle \lambda, X_i\rangle$ and $v_j(\lambda):= \langle \lambda, Z_j\rangle$ be smooth functions on $T^*G$. We have the following expressions for the Poisson brackets
\begin{equation}
\{u_i,v_j\} = \sum_{i=1}^k \sum_{\ell=1}^{n-k} d_{ij}^\ell v_\ell, \qquad i=1,\ldots,k, \quad j=1,\ldots,n-k,
\end{equation}
for some constants $d_{ij}^\ell$. We stress that the above expression does not depend on the $u_i$'s, as a consequence of the graded structure. Denoting the derivative along the integral curves of $\vec{H}$ with a dot, we have
\begin{equation}
\dot{v}_j = \{H,v_j\} = \sum_{i=1}^k u_i \{u_i,v_j\} = \sum_{i=1}^k \sum_{\ell=1}^{n-k} u_i d_{ij}^\ell v_\ell.
\end{equation}
Thus, any integral line of $\vec{H}$ starting from $\lambda \in T^*M^\mathsf{r} = \{v_1=\ldots=v_{n-k} =0\}$ remains in $T^*M^\mathsf{r}$ and the latter is invariant. 

To prove the invariance of $\Theta^\mathsf{r}$, consider, for any fixed left-invariant $X \in \Gamma(\distr)$, the adjoint map $\mathrm{ad}_X : \ver|_e \to \ver|_e$, given by $\mathrm{ad}_X(Z) = [X,Z]|_e$. This map is well defined (as a consequence of the graded structure) and nilpotent. In particular $\tr(\mathrm{ad}_{X}) =0$. Thus, we obtain from an explicit computation (see Appendix~\ref{a:computations})
\begin{equation}
\mathcal{L}_{\vec{H}} \Theta^\mathsf{r} = -\left(\sum_{i=1}^k \sum_{j=1}^{n-k} u_i d_{ij}^j \right) \Theta^\mathsf{r} = -\left(\sum_{i=1}^k u_i\tr(\mathrm{ad}_{X_i})\right) \Theta^\mathsf{r} = 0. \qedhere
\end{equation}
\end{proof}

\begin{prop}[Characterization of reduced geodesics for Carnot groups]
The geodesics $\gamma_\lambda(t)$ with initial covector $\lambda \in T_q^*M^\mathsf{r}$ are obtained by left-translation of straight lines, that is, in exponential coordinates,
\begin{equation}
 \gamma_\lambda(t) = q \star (u t, 0,\ldots,0), \qquad u \in \R^k.
\end{equation}
\end{prop}
\begin{proof}
Let $X_1,\ldots,X_k  \in \Gamma(\distr)$ and $Z_1,\ldots,Z_{n-k} \in \Gamma(\ver)$ be a global frame of left-invariant orthonormal vector fields. Let $u_i(\lambda):= \langle \lambda, X_i\rangle$ and $v_j(\lambda):= \langle \lambda, Z_j\rangle$ be smooth functions on $T^*G$. Let $u \in \R^k$. The extremal $\lambda(t) = \phi_t(\lambda)$, with initial covector $\lambda=(q,u,0)$ satisfies $v \equiv 0$ by Proposition~\ref{p:CarnotOK} and, as a consequence of the graded structure,
\begin{equation}
\dot{u}_i = \{H,u_i\} = \sum_{j=1}^k u_j \{u_j,u_i\} = \sum_{j=1}^k \sum_{\ell=1}^{n-k} u_j c_{ji}^\ell v_{\ell} = 0,
\end{equation}
In particular $\lambda(t) = (q(t),u,0)$. Moreover the geodesic $\gamma_\lambda(t) =\pi(\lambda(t))$ satisfies
\begin{equation}
\dot\gamma_{\lambda}(t) = \sum_{i=1}^k u_i X_i(\gamma_\lambda(t)).
\end{equation}
Since the $u_i$'s are constants, $\gamma_{\lambda}(t)$ is an integral curve of $\sum_{i=1}^k u_i X_i$ starting from $q$. Then $L_{q}^{-1} \gamma_{\lambda}(t)$ is an integral curve of $\sum_{i=1}^k u_i L_{q *}^{-1} X_i =\sum_{i=1}^k u_i X_i$ starting from the identity. By definition of exponential coordinates
\begin{equation}
\gamma_{\lambda}(t) = q \star \exp_G\left(t \sum_{i=1}^k u_i X_i\right) \simeq q \star (ut,0). \qedhere
\end{equation}
\end{proof}

\begin{rmk}
In the case of a step $2$ Carnot group, the group law is linear when written in exponential coordinates. In fact, for a fixed left-invariant basis $X_1,\ldots,X_k \in \Gamma(\distr)$ and $Z_1,\ldots,Z_{n-k} \in \Gamma(\ver)$ it holds
\begin{equation}
[X_i,X_j] = \sum_{\ell=1}^{n-k} c_{ij}^\ell Z_\ell, \qquad c_{ij}^\ell \in \R.
\end{equation}
By the Baker-Campbell-Hausdorff formula, $(x,z) \star (x',z') = (x' + x,z' + z + f(x,x'))$, where
\begin{align}
f(x,x')_\ell = \frac{1}{2} \sum_{i,j=1}^k x_i c_{ij}^\ell x'_j, \qquad \ell =1,\ldots,n-k.
\end{align}
As a consequence, the geodesics $\gamma_\lambda(t)$ with initial covector $\lambda \in T_q^*M^\mathsf{r}$ span the set $q\star \distr|_e$. The latter is not an hyperplane, in general, when $q \neq e$ and the step $m > 2$.
\end{rmk}

\begin{example}[Heisenberg group]
The $(2d+1)$-dimensional Heisenberg group $\mathbb{H}_{2d+1}$ is the sub-Riemannian structure on $\mathbb{R}^{2d+1}$ where $(\distr,g)$ is given by the following set of global orthonormal fields
\begin{equation}
X_i:= \partial_{x_i} - \frac{1}{2} \sum_{i=1}^{2d} J_{ij} x_j \partial_z, \qquad J = \begin{pmatrix}
0 & \mathbb{I}_n \\
-\mathbb{I}_n & 0
\end{pmatrix}, \qquad i=1,\ldots,2d,
\end{equation}
written in coordinates $(x,z) \in \R^{2d} \times \R$. The distribution is bracket-generating, as $[X_i,X_j] = J_{ij} \partial_z$. These fields generate a stratified Lie algebra, nilpotent of step $2$, with
\begin{equation}
\mathfrak{g}_1 = \spn\{X_1,\ldots,X_{2d}\}, \qquad \mathfrak{g}_2 = \spn\{\partial_z\}.
\end{equation}
There is a unique connected, simply connected Lie group $G$ such that $\mathfrak{g} = \mathfrak{g}_1 \oplus \mathfrak{g}_2$ is its Lie algebra of left-invariant vector fields. The group exponential map $\exp_G : \mathfrak{g} \to G$ is a smooth diffeomorphism and then we identify $G = \R^{2d+1}$ with the polynomial product law
\begin{equation}
(x,z) \star (x',z') = \left(x+x',z+z' + \frac{1}{2} x \cdot J x'\right).
\end{equation}
Notice that $X_1,\ldots,X_{2d}$ (and $\partial_z$) are left-invariant.

To carry on the reduction, we consider the Riemannian extension $g$ such that $\partial_{z}$ is a unit vector orthogonal to $\distr$. The geodesics associated with $\lambda \in U^*M^{\mathsf{r}}$ and starting from $q$ reach the whole Euclidean plane $q \star \{z =0 \}$ (the left-translation of $\R^{2d} \subset \R^{2d+1}$). At $q= (x,z)$ this is the plane orthogonal to the vector $\left(\tfrac{1}{2}Jx, 1\right)$ w.r.t.\ the Euclidean metric.
\end{example}

\subsection{Riemannian foliations with bundle like metric}\label{s:foliations}

Roughly speaking, a Riemannian foliation has bundle like metric if locally it is a Riemannian submersion w.r.t.\ the projection along the leaves. 

\begin{definition}\label{def:foliation}
Let $M$ be a smooth and connected $n$-dimensional Riemannian manifold. A $k$-codimensional foliation $\mathcal{F}$ on $M$ is said to be \emph{Riemannian with bundle like metric} if there exists a maximal collection of pairs $\{(U_\alpha,\pi_\alpha), \alpha \in I\}$ of open subsets $U_\alpha$ of $M$ and submersions $\pi_\alpha : U_\alpha \to U_\alpha^0 \subset \mathbb{R}^k$ such that
\begin{itemize}
\item $\{U_\alpha\}_{\alpha \in I}$ is a covering of $M$
\item If $U_\alpha \cap U_\beta \neq \emptyset$, there exists a local diffeomorphism $\Psi_{\alpha\beta} : \R^k \to \R^k$ such that $\pi_\alpha = \Psi_{\alpha\beta} \pi_\beta$ on $U_\alpha \cap U_\beta$
\item the maps $\pi_\alpha: U_\alpha \to U_\alpha^0$ are Riemannian submersions when $U_\alpha^0$ are endowed with a given Riemannian metric 
\end{itemize}
On each $U_\alpha$, the preimages $\pi_\alpha^{-1}(x_0)$ for fixed $x_0 \in U_\alpha^0$ are codimension $k$ embedded submanifolds, called the \emph{plaques} of the foliation. These submanifolds form maximal connected injectively immersed submanifolds called the \emph{leaves} of the foliation. The foliation is \emph{totally geodesic} if its leaves are totally geodesic submanifolds \cite{Tondeur}.
\end{definition}

To any Riemannian foliation with bundle-like metric we associate the splitting $T M = \distr \oplus \ver$, where $\ver$ is the bundle of vectors tangent to the leaves of the foliation and $\distr$ is its orthogonal complement (we call $\ver$ the bundle of \emph{vertical directions}, and its sections \emph{vertical vector fields}). If $\distr$ is bracket-generating, then $(\distr,g|_{\distr})$ is indeed a sub-Riemannian structure on $M$ that we refer to as \emph{tamed by a foliation} and we assume to be equipped with the corresponding Riemannian volume.

We say that a vector field $X \in \Gamma(TM)$ is basic if, locally on any $U_\alpha$, it is $\pi_\alpha$-related with some vector $X^0$ on $U_\alpha^0$. If $X \in \Gamma(TM)$ is basic, and $V \in \Gamma(\ver)$ is vertical, then the Lie bracket $[X,V]$ is vertical. In this setting we consider a local orthonormal frame $Z_1,\ldots,Z_{n-k} \in \Gamma(\ver)$ of vertical vector fields and a local orthonormal frame of basic vector fields $X_1,\ldots,X_k \in \Gamma(\distr)$ for the distribution. The structural functions are defined as
\begin{equation}
[X_i,X_j] = \sum_{\ell=1}^k b_{ij}^\ell X_\ell + \sum_{\ell = 1}^{n-k} c_{ij}^\ell Z_\ell, \qquad [X_i,Z_j] = \sum_{\ell=1}^{n-k} d_{ij}^\ell Z_\ell, \qquad [Z_i,Z_j] = \sum_{\ell=1}^{n-k} e_{ij}^\ell Z_\ell.
\end{equation}
The totally geodesic assumption is equivalent to the fact that any basic horizontal vector field $X$ generates a vertical isometry, that is
\begin{equation}\label{eq:skew}
(\mathcal{L}_X g)(Z,W) = 0, \qquad \forall Z,W \in \Gamma(\ver)\qquad \iff \qquad d_{ij}^\ell = -d_{i\ell}^j.
\end{equation}

\begin{prop}\label{p:foliationsOK}
Any sub-Riemannian structure tamed by a foliation with totally geodesic leaves satisfies assumptions (\textbf{H1}) and (\textbf{H2}).
\end{prop}

\begin{proof}
	Locally, $T^*M^\mathsf{r}$ is the zero-locus of the functions $v_i(\lambda) = \langle \lambda, Z_i \rangle$ for some family $\{Z_j\}_{j=1}^{n-k}$ of generators of $\ver$. Thus, denoting the derivative along the integral curves of $\vec{H}$ with a dot, we have
	\begin{equation}
	\dot{v}_j = \{H, v_i\} = \sum_{i=1}^k u_i \{u_i,v_j\} = \sum_{i=1}^k \sum_{\ell=1}^{n-k} u_i d_{ij}^\ell v_\ell =0, \qquad \text{on $T^*M^\mathsf{r}$}.
	\end{equation}
This readily implies the invariance of $T^*M^\mathsf{r}$.  To prove the invariance of $\Theta^\mathsf{r}$, we obtain from an explicit computation (see Appendix~\ref{a:computations})
\begin{equation}
\mathcal{L}_{\vec{H}}(\Theta^\mathsf{r}) =  -\left(\sum_{i=1}^k \sum_{\ell=1}^{n-k} u_i d_{i \ell}^\ell\right) \Theta^\mathsf{r} =0,
\end{equation}
where, in the last step, we used the totally geodesic assumption~\eqref{eq:skew}.
\end{proof}

\subsubsection{Riemannian submersions}

A Riemannian submersion $ \pi : (M,g) \to (\bar{M},\bar{g})$ is trivially a Riemannian foliation with bundle-like metric. Let $M$ be a sub-Riemannian manifold tamed by a Riemannian submersion $\pi : M \to \bar{M}$. We have the following characterization.
\begin{prop}
Let $M$ be a sub-Riemannian manifold tamed by a Riemannian submersion $\pi : M \to \bar{M}$. Then $\gamma_{\lambda}:[0,T] \to M$ is a sub-Riemannian geodesic associated with $\lambda \in U^*M^\mathsf{r}$ if and only if it is the lift of a Riemannian geodesic $\bar\gamma_{\lambda}:=\pi \circ \gamma_{\lambda}$ of $\bar{M}$.
\end{prop}
\begin{proof}
Let $\bar{X}_1,\ldots,\bar{X}_k \in \Gamma(T\bar{M})$ be a local orthonormal frame for $(\bar{M},\bar{g})$. Let $X_1,\ldots,X_{k} \in \Gamma(\distr)$ the corresponding local orthonormal frame of basic vector fields on $M$, such that $\pi_* X_i = \bar{X}_i$. Let $Z_1,\ldots,Z_{n-k} \in \Gamma(\ver)$ be a local orthonormal frame for $\ver$. Indeed
\begin{equation}
[X_i,X_j] = \sum_{\ell=1}^{k} b_{ij}^\ell X_\ell + \sum_{\ell=1}^{n-k} c_{ij}^\ell Z_\ell, \qquad b_{ij}^\ell, c_{ij}^\ell \in C^\infty(M).
\end{equation}
Since the $X_i$'s are basic, the functions $b_{ij}^\ell \in C^\infty(M)$ are constant along the fibers of the submersion and descend to well defined functions in $C^\infty(\bar{M})$. Moreover
\begin{equation}
[\bar{X}_i,\bar{X}_j] = \sum_{\ell=1}^{k} b_{ij}^\ell \bar{X}_\ell .
\end{equation}
Sub-Riemannian extremals $\lambda(t) \in U^*M^\mathsf{r}$ satisfy
\begin{equation}\label{eq:eqsR}
v_j(t) \equiv 0,\qquad \dot{u}_j(t) = \sum_{i,\ell=1}^k u_i(t) b_{ij}^\ell u_\ell(t), \qquad \dot{\gamma}_{\lambda}(t) = \sum_{i=1}^k u_i(t) X_i(\gamma_\lambda(t)),
\end{equation}
where the structural functions $b_{ij}^\ell = b_{ij}^\ell(\gamma_{\lambda}(t))$ are computed along the sub-Riemannian geodesic.
On the other hand, Riemannian extremals $\bar{\lambda}(t) \in U\bar{M}$ satisfy
\begin{equation}\label{eq:eqR}
\dot{\bar{u}}_j(t) = \sum_{i,\ell=1}^k \bar{u}_i(t) b_{ij}^\ell \bar{u}_\ell(t), \qquad \dot{\gamma}_{\lambda}(t) = \sum_{i=1}^k \bar{u}_i(t) \bar{X}_i(\bar{\gamma}_\lambda(t)),
\end{equation}
where $\bar{u}_i :T^*\bar{M} \to \R$ are the smooth functions $u_i(\bar{\lambda}) = \langle \bar\lambda,\bar{X}_i\rangle$, for $i=1,\ldots,k$ and are computed along the extremal. The statement follows by observing that the projections $\bar\gamma_\lambda = \pi \circ \gamma_{\lambda}$ of sub-Riemannian extremals satisfy~\eqref{eq:eqR} with $\bar{u}_i(t) = u_i(t)$. Viceversa, for any Riemannian geodesic $\bar{\gamma}_{\bar\lambda}$ on $\bar{M}$, its horizontal lift $\gamma_{\lambda}$ on $M$ satisfies~\eqref{eq:eqsR} with $u_i(t) = \bar{u}_i(t)$ and $v_j \equiv 0$. 
\end{proof}

\begin{example}[Complex Hopf fibrations]\label{ex:cr-hopf}
Consider the odd dimensional spheres $\mathbb{S}^{2d+1}$
\begin{equation}
\mathbb{S}^{2d+1} = \{(z_0,z_1,\ldots,z_d) \in \C^{d+1} \mid \|z\| = 1 \},
\end{equation}
 equipped with the standard round metric. The unit complex numbers $\mathbb{S}^1 = \{z \in \C \mid |z|=1\}$ give an isometric action of $\mathrm{U}(1)$ on $\mathbb{S}^{2d+1}$ by
\begin{equation}
z  \to e^{i\theta} z, \qquad z \in \mathbb{S}^{2d+1},\, \theta \in (-\pi,\pi].
\end{equation}
Hence, the quotient space $\mathbb{S}^{2d+1}/\mathbb{S}^1 \simeq \mathbb{CP}^d$ (the complex projective space) has a unique Riemannian structure (the Fubini-Study metric) such that the projection
\begin{equation}
p(z_0,\ldots,z_d) = [z_0:\ldots:z_d]
\end{equation}
is a Riemannian submersion. The fibration $\mathbb{S}^1 \hookrightarrow \mathbb{S}^{2d+1} \xrightarrow{p} \mathbb{CP}^d$ is called the \emph{complex Hopf fibration}. In real coordinates $z_j = x_j + i y_j$ on $\C^{d+1}$, the vertical distribution $\ver = \ker p_*$ is generated by the restriction to $\mathbb{S}^{2d+1}$ of the unit vector field
\begin{equation}
\xi = \sum_{j=0}^d (x_j \partial_{y_j} - y_j \partial_{x_j}).
\end{equation}
The orthogonal complement $\distr := \ver^\perp$ with the restriction $g|_{\distr}$ of the round metric define the standard sub-Riemannian structure on the complex Hopf fibrations. 
In real coordinates, as subspaces of $\R^{2d+2}$, the hemisphere and its boundary are
\begin{equation}
M=\mathbb{S}_+^{2d+1}:= \left\lbrace \sum_{i=0}^d x_i^2+y_i^2 =1 \mid x_0 \geq 0  \right\rbrace, \quad \partial M = \left\lbrace \sum_{i=0}^d x_i^2+y_i^2 =1 \mid x_0 = 0  \right\rbrace.
\end{equation}

A different set of coordinates we will use is the following
\begin{equation}
(\theta,w_1,\ldots,w_d) \mapsto \left(\frac{e^{i\theta}}{\sqrt{1+|w|^2}},\frac{w_1 e^{i\theta}}{\sqrt{1+|w|^2}},\ldots,\frac{w_d e^{i\theta}}{\sqrt{1+|w|^2}}  \right),
\end{equation}
where $\theta \in (-\pi,\pi)$ and $w=(w_1,\ldots,w_d) \in \C^d$. In particular $(w_1,\ldots,w_d)$ are inohomgeneous coordinates for $\mathbb{CP}^d$ given by $w_j = z_j/z_0$ and $\theta$ is the fiber coordinate. The north pole corresponds to $\theta = 0$ and $w = 0$. The hemisphere is characterized by $\theta \in [-\tfrac{\pi}{2},\tfrac{\pi}{2}]$ and its boundary by $\cos(\theta) = 0$.
\end{example}

\begin{example}[Quaternionic Hopf fibrations]\label{ex:qhf}
Let $\H$ be the field of quaternions. If $q = x+ iy +jz + kw$, with $x,y,z,w \in \R$, the quaternionic norm is
\begin{equation}
\|q\| = x^2 + y^2 + z^2 +w^2.
\end{equation}
Consider the sphere $\mathbb{S}^{4d+3}$ as a subset of the quaternionic space $\H^d$,
\begin{equation}
\mathbb{S}^{4d+3} = \{(q_0,q_1,\ldots,q_d) \in \H^{d+1} \mid \|q\| = 1 \},
\end{equation}
equipped with the standard round metric. The left multiplication by unit quaternions $\mathbb{S}^3 = \{q \in \H \mid |q|=1\}$ gives an isometric action of $\mathrm{SU}(2)$ on $\mathbb{S}^{4d+3}$. The quotient space $\mathbb{S}^{4d+3}/\mathbb{S}^3 \simeq \mathbb{HP}^d$ (the quaternionic projective space) has a unique Riemannian structure such that the projection
\begin{equation}
p(q_0,\ldots,q_d) = [q_0:\ldots:q_d]
\end{equation}
is a Riemannian submersion. The fibration $\mathbb{S}^3 \hookrightarrow \mathbb{S}^{4d+3} \xrightarrow{p} \mathbb{HP}^d$ is the \emph{quaternionic Hopf fibration}. In real coordinates $q_j = x_j + i y_j + j z_j + k w_j$ on $\H^{d+1}$, the vertical distribution $\ver = \ker p_*$ is generated by
\begin{gather}
\xi_I   = \sum_{i=0}^{d}  y_i\partial_{x_i} - x_i\partial_{y_i}  +w_i\partial_{z_i} -z_i\partial_{w_i}, \qquad \xi_J  = \sum_{i=0}^{d}    z_i\partial_{x_i} -w_i\partial_{y_i} -x_i\partial_{z_i} +y_i\partial_{w_i},\\
\xi_K  = \sum_{i=0}^{d} w_i\partial_{x_i} + z_i\partial_{y_i}  -y_i\partial_{z_i} - x_i\partial_{w_i}.   
\end{gather}
The orthogonal complement $\distr := \ver^\perp$ with the restriction $g|_{\distr}$ of the round metric define the standard sub-Riemannian structure on the quaternionic Hopf fibrations.

In real coordinates, the hemisphere $M  =\mathbb{S}_+^{4d+4} \subset \R^{4d+4}$ and its boundary are
\begin{align}
M & = \left\lbrace \sum_{i=0}^d x_i^2+y_i^2+z_i^2+w_i^2 =1 \mid x_0 \geq 0  \right\rbrace, \\
\partial M  &= \left\lbrace \sum_{i=0}^d x_i^2+y_i^2+z_i^2+w_i^2 =1 \mid x_0 = 0  \right\rbrace.
\end{align}

A different set of coordinates we will use is the following
\begin{equation}
(\theta_1,\theta_2,\theta_3,w_1,\ldots,w_d) \mapsto \left(\frac{e^{i\theta_1+j\theta_2 + k\theta_3}}{\sqrt{1+|w|^2}},\frac{w_1 e^{i\theta_1+j\theta_2 + k\theta_3}}{\sqrt{1+|w|^2}},\ldots,\frac{w_d e^{i\theta_1+j\theta_2 + k\theta_3}}{\sqrt{1+|w|^2}}  \right),
\end{equation}
where $|\theta|^2 =\theta_1^2+ \theta_2^2 +\theta_3^2 < \pi^2$ and $w=(w_1,\ldots,w_d) \in \mathbb{H}^d$. In particular $(w_1,\ldots,w_d)$ are inohomgeneous coordinates for $\mathbb{HP}^d$ given by $w_j = q_0^{-1} q_j$ and $\theta_1,\theta_2,\theta_3$ are local coordinates on $\mathrm{SU}(2)$. The north pole corresponds to $\theta_1=\theta_2=\theta_3 = 0$ and $w = 0$. The hemisphere is characterized by $|\theta| \leq \pi/2$ and its boundary by $\cos|\theta| = 0$.
\end{example}

\section{Applications}\label{s:applications}

In this section we present the proofs of the applications of the reduced Santal\'o formula presented in Sections~\ref{sec:hardy-type-intro}, \ref{sec:spectral-gap-intro}  and \ref{sec:isoperimetric-intro}.

\subsection{Hardy-type inequalities}
It is well known that for all $f\in C_0^\infty([0,a])$ one has
\begin{equation}\label{eq:poincare-1d}
  \int_0^a f'(t)^2 dt \ge
  \frac{\pi^2}{a^2} \int_0^a f(t)^2 dt, \qquad \text{(1D Poincar\'e inequality)}
\end{equation}
with equality holding if and only if $f(t) = C \sin\left(\frac\pi{a} t\right)$. Moreover,
\begin{equation}\label{eq:hardy-1d}
  \int_0^a f'(t)^2 dt \ge
  \frac{1}{4} \int_0^a \frac{f(t)^2}{d(t)^2} dt, \qquad \text{(1D Hardy inequality)}
\end{equation}
where $d(t) = \min\{t,a-t\}$ is the distance from the boundary and the equality holds if and only if $f(t) = 0$.

Recall that $\ell(\lambda)$ is the length at which the geodesic with initial covector $\lambda$ leaves $M$ crossing the boundary $\partial M$  and, in general, $t \mapsto \ell(\phi_t(\lambda))$ is a decreasing function. Then $\ell(\cdot)$ is not invariant under the flow $\phi_t$. For this reason in Section~\ref{sec:hardy-type-intro} we introduced the function $L :U^*M \to [0,+\infty]$ defined as $L(\lambda) := \ell(\lambda) + \ell(-\lambda)$, that measures the length of the projection of the maximal integral line of $\vec{H}$ passing through $\lambda$.
Indeed, $L(\cdot)$ is $\phi_t$-invariant and coincides with $\ell(\cdot)$ on $U^+\partial M$, since it can be equivalently defined as
\begin{equation}
L(\lambda):= \begin{cases}
\ell(-\phi_{\ell(-\lambda)}(-\lambda)), &  \text{if } \ell(-\lambda) < + \infty, \\
+\infty, & \text{otherwise}.
\end{cases}
\end{equation} 

\begin{proof}[Proof of Proposition~\ref{p:hardylike-intro}]

Choose coordinates $x$ around a fixed $q \in M$ as in Lemma~\ref{l:eta-r}, and let $(p,x)$ be the associated canonical coordinates on $T^*M$. Let $Q$ be a quadratic form on $T_{q}^*M^\mathsf{r}$. In particular $Q(\lambda) =  \sum_{i,j=1}^k  p_iQ_{ij} p_j$, where $\lambda = (p_1,\ldots,p_k)$. By Lemma~\ref{l:eta-r} we have
\begin{equation}
\int_{U_{q}^*M^\mathsf{r}} Q(\lambda) \eta_{q_0}^\mathsf{r}(\lambda) = \int_{\mathbb{S}^{k-1}} \sum_{i,j=1}^k Q_{ij} p_i p_j\, d \mathrm{vol}_{\mathbb{S}^{k-1}}(p) = \frac{|\mathbb{S}^{k-1}|}{k} \tr(Q),
\end{equation}
where we performed the standard integral of a quadratic form on $\mathbb{S}^{k-1}$. Choosing $Q(\lambda) = \langle \lambda, \nabla_H f(q)\rangle^2$, then $\tr(Q) = |\nabla_H f(q)|^2$. Thus, for any point $q \in M$ we have
\begin{equation}\label{eq:104}
\frac{|\mathbb S^{k-1}|}{k} |\nabla_H f(q)|^2 =\int_{U_q^* M^\mathsf{r}} \langle \lambda, \nabla_H f(q)\rangle^2 \eta_q^\mathsf{r}(\lambda), \qquad \forall f \in C^\infty(M).
\end{equation}
Using the reduced Santal\'o formula \eqref{eq:santalo-reduced},
\begin{equation*}
\begin{split}
  \frac{|\mathbb S^{k-1}|}{k} \int_M |\nabla_H f(q)|^2 \omega(q) 
    &= \int_M  \left[\int_{U_q^* M^\mathsf{r}} \langle \lambda, \nabla_H f(q)\rangle^2 \eta_q^\mathsf{r}(\lambda) \right]\omega(q)\\
    &= \int_{U^*M^\mathsf{r}} \langle \lambda, \nabla_H f\rangle^2 \mu^\mathsf{r} \\
    &\ge \int_{U^\vis M^\mathsf{r}} \langle \lambda, \nabla_H f\rangle^2 \mu^\mathsf{r} \\
    & = \int_{\partial M} \left[\int_{U_q^+\partial M^\mathsf{r}}  \left( \int_0^{\ell(\lambda)}  \langle \phi_t(\lambda), \nabla_H f\rangle^2 dt\right)  \langle \lambda,{\mathbf{n}_q}\rangle \eta^\mathsf{r}_q(\lambda)\right] \sigma(q).
\end{split}
\end{equation*}
Consider the subset $D = U^+\partial M^\mathsf{r} \cap \{\ell < +\infty\}$. Let $f_\lambda(t) := f(\pi\circ\phi_t(\lambda))$. For $\lambda \in D$ we have $f_\lambda(0) = f_\lambda(\ell(\lambda)) = 0$ and the one-dimensional Poincar\'e inequality \eqref{eq:poincare-1d} gives
\begin{equation}\label{eq:stepineq}
\int_0^{\ell(\lambda)} \langle \phi_t(\lambda), \nabla_H f\rangle^2  dt 
 =  \int_0^{\ell(\lambda)}  f'_\lambda(t)^2  dt \ge \frac{\pi^2}{\ell^2(\lambda)}\int_0^{\ell(\lambda)}  f_\lambda(t)^2  dt.
\end{equation}
Indeed we can replace $\ell$ with $L$, which is $\phi_t$-invariant. Then
\begin{equation}
\frac{|\mathbb S^{k-1}|}{k} \int_M |\nabla_H f(q)|^2 \omega(q) 
\geq \pi^2\int_{\partial M} \left[\int_{D_q}  \left( 
\int_0^{\ell(\lambda)}  \frac{f_\lambda(t)^2}{L(\lambda)^2} dt
\right) \langle \lambda,{\mathbf{n}_q}\rangle \eta^\mathsf{r}_q(\lambda)\right] \sigma(q).
\end{equation}
Since on $U^+_q\partial M \setminus D_q$ the function $1/L(\lambda)^2 = 0$, we can replace $D_q$ with $U_q^+\partial M$. Using again Santal\'o formula to restore the integral on $U^\vis M^\mathsf{r}$, we obtain
\begin{equation}
\int_M |\nabla_H f(q)|^2 \omega(q) 
 \geq  \frac{k \pi^2}{|\mathbb S^{k-1}| } \int_{U^\vis M^\mathsf{r}} \frac{(\pi^* f)^2}{L^2} \mu^\mathsf{r}
 = \frac{k \pi^2}{|\mathbb{S}^{k-1}|}\int_M \left[\int_{U_q^* M^\mathsf{r}}\frac{1}{L^2}\eta^\mathsf{r}_q\right] f(q)^2 \omega(q).
\end{equation}
The second equality follows by Lemma~\ref{l:vert-split-sub}. This concludes the proof of~\eqref{eq:hardy1-intro}.

To prove~\eqref{eq:hardy2-intro} we replace Poincar\'e inequality with Hardy in~\eqref{eq:stepineq}:
\begin{multline}
\int_0^{\ell(\lambda)} \langle \phi_t( \lambda), \nabla_H f\rangle^2  dt 
 =  \int_0^{\ell(\lambda)}  f'_\lambda(t)^2  dt \\ 
 \ge \frac{1}{4}\int_0^{\ell(\lambda)}  \frac{f_\lambda(t)^2}{\min\{t,\ell(\lambda)-t\}^2}  dt 
 \ge \frac{1}{4}\int_0^{\ell(\lambda)}  \frac{f_\lambda(t)^2}{\ell(\phi_t(\lambda))^2}  dt,
\end{multline}
where we used the fact that if $\lambda \in U^+\partial M$ then $\ell(\phi_t(\lambda)) = \ell(\lambda)-t$.
We then proceed as in the previous case without replacing $\ell$ with $L$.
\end{proof}

\begin{proof}[Proof of Proposition~\ref{p:p-hardy-intro}]
The result is obtained by mimicking the proof of Proposition \ref{p:hardylike-intro}. Observe that, for any $f \in C^\infty_0(M)$ and $q \in M$, we have
\begin{align}
\int_{U_q^* M^\mathsf{r}} |\langle \lambda, \nabla_H f(q)\rangle|^p \eta_q^\mathsf{r}(\lambda) & = |\nabla_H f(q)|^p \int_{U_q^* M^\mathsf{r}} \left|\langle \lambda, \frac{\nabla_H f(q)}{|\nabla_H f(q)| }\rangle\right|^p \eta_q^\mathsf{r}(\lambda)  \\
& = 2 |\nabla_H f(q)|^p  \int_{\mathbb{S}^{k-1}\cap \{p_1 > 0\}}p_1^p \,d\mathrm{vol}_{\mathbb{S}^{k-1}}(p) \\
& = |\nabla_H f(q)|^p C_{p,k}^{-1},
\end{align}
where we used coordinates as in Lemma~\ref{l:eta-r}, the rotational invariance of the measure and
\begin{equation}
C_{p,k} := \frac{\Gamma(\tfrac{k+p}{2})}{2\Gamma(\tfrac{1+p}{2})\pi^{(k-1)/2}} = \frac{k}{|\mathbb{S}^{k-1}|} \frac{\sqrt{\pi}\,\Gamma(\tfrac{k+p}{2})}{2\Gamma(\tfrac{1+p}{2})\Gamma(\tfrac{k}{2}+1)}.
\end{equation}
Using the above in place of~\eqref{eq:104}, by Santal\'o formula~\eqref{eq:santalo-reduced} we obtain
\begin{equation}
C_{p,k}^{-1} \int_M |\nabla_H f|^p \omega(q) \geq \int_{\partial M} \left[\int_{U_q^+\partial M^\mathsf{r}}  \left( \int_0^{\ell(\lambda)}  |\langle \phi_t(\lambda), \nabla_H f\rangle|^p dt\right)  \langle \lambda,{\mathbf{n}_q}\rangle \eta^\mathsf{r}_q(\lambda)\right] \sigma(q).
\end{equation}
To prove~\eqref{eq:p-hardy1-intro} we proceed as in the proof of Proposition~\ref{p:hardylike-intro} replacing the step~\eqref{eq:stepineq} with the $L^p$ Poincar\'e inequality \cite[Sec. 5.3]{lindq}
\begin{equation}
\int_0^a |f'(t)|^p dt \ge  \left(\frac{\pi_p}{a}\right)^p \int_0^a |f(t)|^p dt.
\end{equation}
We proceed similarly for the proof of~\eqref{eq:p-hardy2-intro}, replacing \eqref{eq:stepineq} with the $L^p$ Hardy's inequalities \cite[Thm. 327]{HLP-Inequalities}
\begin{equation}
\int_0^a |f'(t)|^p dt \ge
  \left(\frac{p-1}{p}\right)^p \int_0^a \frac{|f(t)|^p}{d(t)^p} dt.
  \qedhere
\end{equation}
\end{proof}

\begin{proof}[Proof of Proposition~\ref{p:lowerbound-intro}] With $L := \sup_{\lambda\in U^*M^{\rm r}} L(\lambda)$, the Hardy inequality \eqref{eq:hardy1-intro} can be further simplified into
\begin{equation}
\int_M |\nabla_H f|^2 \omega \geq \frac{k \pi^2}{L^2}\int_M f^2 \omega.
\end{equation}
By the min-max principle~\eqref{eq:minmax}, whenever any $f \in C_0^\infty(M)$ such that $\int_M f^2 \omega =1$, we have
\begin{equation}
\lambda_1(M) \geq \int_M |\nabla_H f|^2 \omega \geq \frac{k \pi^2}{L^2}. \qedhere
\end{equation}
\end{proof}

\begin{proof}[Proof of Proposition~\ref{p:sharpness-lowerbound-intro}]
Fix a north pole $q_0$ and the hemisphere $M$ whose center is $q_0$. By Remark~\ref{r:twoflows}, in all three cases, the reduced sub-Riemann\-ian geodesics are a subset of the Riemann\-ian ones (great circles). In particular $L = \pi$ and Proposition~\ref{p:lowerbound-intro} gives
\begin{equation}
\lambda_1(M) \geq k,
\end{equation}
where $k = d$ for the Riemannian sphere, $k= 2d$ for the CHF and $k=4d$ for the QHF.

In all cases, uniqueness of $\Phi=\cos(\delta)\in C^\infty_0(M)$ follows as in the Riemannian case from the min-max principle \cite[Corollary~2, p.\ 20]{chavelbook-eigen}.
To complete the proof, we show that $\Phi$ is an eigenfunction of the positive (sub-)Laplacian with eigenvalue $d$, $2d$, $4d$, respectively.  In the Riemannian case this is well known. For the CHF we use coordinates $(\theta,w)$ of Example~\ref{ex:cr-hopf}. Then,
\begin{equation}
\Phi =x_0 = \frac{\cos(\theta)}{\sqrt{1+|w|^2}} = \cos(\theta)\cos(r),
\end{equation}
where we have set $\tan(r) = |w|$. In \cite[Proposition 2.3]{BW-CR} the authors show that for a function depending only on $\theta$ and $r$ the action of the sub-Laplacian reduces to the action of its cylindrical part, given by
\begin{equation}
\tilde{\Delta} = \partial_r^2 + ((2d-1)\cot(r) - \tan(r)) \partial_r + \tan^2(r)\partial^2_\theta.
\end{equation}
In particular $\Delta \Phi = \tilde\Delta \Phi = (-2d) \Phi$.

For the QHF we use coordinates $(\theta_1,\theta_2,\theta_3,w)$ of Example~\ref{ex:qhf}. Using the expression
\begin{equation}
e^{i\theta_1+j\theta_2 +k \theta_3} = \cos(\eta) + (i\theta_1 + j \theta_2 + k \theta_3)\frac{\sin(\eta)}{\eta} , \qquad \eta: = \sqrt{\theta_1^2+\theta_2^2+\theta_3^2},
\end{equation}
we obtain
\begin{equation}
\Phi = x_0 = \frac{\cos(\eta)}{\sqrt{1+|w|^2}} = \cos(\eta)\cos(r),
\end{equation}
where we have set $\tan(r) = |w|$. In \cite[Definition 2.1 and Proposition 2.2]{BW-QHF} the authors show that, for a function depending only on $\eta$ and $r$, the action of the sub-Laplacian reduces to the action of its cylindrical part, given by
\begin{equation}
\tilde{\Delta} = \partial_r^2 + ((4d-1)\cot(r) - 3\tan(r)) \partial_r + \tan^2(r)\left(\partial^2_\eta + 2\cot(\eta)\partial_\eta\right).
\end{equation}
In particular $\Delta \Phi = \tilde\Delta \Phi = (-4d) \Phi$.
\end{proof}

\subsection{Isoperimetric inequalities} \label{sub:isoperimetric_inequalities}

We define some quantities that we already introduced.

\begin{definition}
	\label{def:vis-angle}
	The \emph{visibility angle at $q\in M$} and the \emph{optimal visibility angle} are
	\begin{equation}
		\theta^\vis_q :=  \frac{\eta_q^\mathsf{r}(U_q^\vis M^\mathsf{r})}{ \eta_q^\mathsf{r}(U_q M^\mathsf{r})} ,
		\qquad \tilde\theta^\vis_q := \frac{\eta_q^\mathsf{r}(\tilde U_q^\vis M^\mathsf{r})}{\eta_q^\mathsf{r}(U_q M^\mathsf{r})}.
	\end{equation}
	The \emph{least visibility angle} and the \emph{least optimal visibility angle} are
	\begin{equation}
	 	\theta^\vis = \inf_{q\in M}\theta^\vis_q , \qquad \tilde\theta^\vis = \inf_{q\in M}\tilde\theta^\vis_q.
	 \end{equation} 
\end{definition}

Notice that $\theta^\vis_q, \tilde \theta^\vis_q, \theta^\vis, \tilde\theta^\vis  \in [0,1]$ and do not depend on the choice of the volume $\omega$.

\begin{definition}
The sub-Riemannian \emph{diameter} and \emph{reduced diameter} are:
\begin{align}
\diam(M) & := \sup\{\d(x,y)\mid x,y \in M\} = \sup\{\tilde{\ell}(\lambda) \mid \lambda \in U^*M \}, \\
\diam^\mathsf{r}(M) & := \sup\{\tilde{\ell}(\lambda) \mid \lambda \in U^*M^\mathsf{r} \}.
\end{align}
\end{definition}
Clearly $\diam^{\mathsf{r}}(M) \leq \diam(M)$.

\begin{proof}[Proof of Proposition~\ref{p:type1iso-intro}]
	The proof follows as in \cite{chavelbook, croke-iso}, considering $F=1$ in~\eqref{eq:santalo-reduced}.
	The l.h.s.\ is estimated from below using the disintegration of $\mu^\mathsf{r}$ given in Lemma~\ref{l:vert-split-sub}.
	For the estimate of the r.h.s.\ we only observe that, by Lemma~\ref{l:eta-r} we have
	\begin{equation}
		\int_{U^+_q\partial M^\mathsf{r}} \langle \lambda, \mathbf{n}_q \rangle \eta_q^\mathsf{r}(\lambda) 
		= \int_{\mathbb{S}^{k-1}\cap\{p_1>0\}} p_1\, d\mathrm{vol}_{\mathbb S^{k-1}}(p)
		= \frac{|\mathbb{S}^k|}{2\pi}. \qedhere
	\end{equation}
\end{proof}

\begin{proof}[Proof of Proposition~\ref{p:sharp-iso-intro}]
For all these structures, all the inequalities in the proof of Proposition~\ref{p:type1iso-intro} are equalities, hence the sharpness follows. Anyway, here we perform the explicit computation for the hemisphere of the sub-Riemannian complex Hopf fibration; the remaining case of the quaternionic Hopf fibration can be checked following the same steps. 

We use the notation of Example~\ref{ex:cr-hopf}, and real coordinates.
Let $q=(0,y_0,\ldots,x_d,y_d) \in \partial M$. The sub-Riemannian normal $\mathbf{n}_q$ is the unique inward pointing unit vector in $\distr_q$ orthogonal to $\distr_q \cap T_q \partial M$. Indeed, $T_q \partial M$ is the orthogonal complement to $\partial_{x_0}$ w.r.t.\ to the Riemannian round metric, while $\distr_q$ is the orthogonal complement to $\xi$. Thus, $\mathbf{n}_q=\alpha \xi + \beta \partial_{x_0}$. The condition $\mathbf{n}_q \in \distr$ and the normalization imply 
\begin{equation}
\mathbf{n}_q = \frac{1}{\sqrt{1-g(\partial_{x_0},\xi)^2}}\left(\partial_{x_0} - g(\partial_{x_0},\xi) \xi \right).
\end{equation}
Using the explicit expression for $\xi$ we obtain
\begin{equation}
\mathbf{n}_q = \sqrt{1-y_0^2}\,\partial_{x_0} \mod T_q \partial M.
\end{equation}
Notice that $C(\partial M) = \{ y_0^2 = 1\} \cap \partial M$. Due to the factor $\rho(y_0):=\sqrt{1-y_0^2}$, the sub-Riemannian surface measure $\sigma$ is different from the Riemannian one in cylindrical coordinates $\sigma_R = \iota_{\partial_{x_0}} \omega = \rho^{2d-2} dy_0\, d\mathrm{vol}_{\mathbb{S}^{2d-1}}$. In particular
\begin{equation}
\sigma(\partial M) = \int_{\partial M} \rho(y_0)\, \sigma_R = |\mathbb{S}^{2d-2}|  \int_{-1}^1  \rho(y_0)^{2d-1} d y_0 = \frac{|\mathbb{S}^{2d+1}|}{|\mathbb{S}^{2d}|}.
\end{equation}
Moreover $\omega(M) = |\mathbb{S}^{2d+1}|/2$. By Remark~\ref{r:twoflows}, the reduced geodesics $\gamma_\lambda$ with $\lambda \in U^*M^\mathsf{r}$ are a subset of Riemannian geodesics hence $\vartheta^\vis = \tilde{\vartheta}^\vis = 1$ and $\ell = \diam^\mathsf{r}(M) = \pi$.
\end{proof}

\appendix

\section{Lie derivative of the reduced Liouville volume}\label{a:computations}

\begin{lemma}
In the notation of Section~\ref{sec:redsant}, we have
\begin{equation}
\mathcal{L}_{\vec{H}} \Theta^\mathsf{r} = - \left( \sum_{j=1}^{n-k} \sum_{i=1}^k  u_i \partial_{v_j} \{u_i,v_j\}\right) \Theta^\mathsf{r}.
\end{equation}
\end{lemma}
\begin{proof}
For any $\ell$-tuple $w = (w_1,\ldots,w_{\ell})$ of vector fields and $\ell$-form $\alpha$, we denote
\begin{equation}
\alpha(\mathcal{L}_{\vec{H}}(w)) = \sum_{i=1}^{\ell} \alpha(w_1,\ldots,[\vec{H},w_i],\ldots,w_\ell).
\end{equation}
Let $(x_1,\ldots,x_n): U \to \R^n$ be coordinates on $U \subset M$. Then $(x,u,v): \pi^{-1}(U) \to \R^{2n}$ are local coordinates for $T^*M$ and $T^*M^\mathsf{r} \cap \pi^{-1}(U) = \{(x,u,v)\mid v = 0 \}$.

Denote $\partial_v = (\partial_{v_1},\ldots,\partial_{v_{n-k}})$, $\partial_u = (\partial_{u_1},\ldots,\partial_{u_k})$ and $\partial_x = (\partial_{x_1},\ldots,\partial_{x_n})$. 
Recall that $\Theta^\mathsf{r}(\partial_u,\partial_x) = \Theta(\partial_u,\partial_v,\partial_x) = (-1)^{k(n-k)} \Theta(\partial_v,\partial_u,\partial_x)$. Then, using twice $(\mathcal{L}_{\vec{H}} \alpha)(w) = \vec{H}(\alpha(w)) -\alpha(\mathcal{L}_{\vec{H}}(w))$ for any $\ell$-form $\alpha$ and $\ell$-uple $w$, we obtain
\begin{align}
(\mathcal{L}_{\vec{H}} \Theta^\mathsf{r})(\partial_u,\partial_x) & = \vec{H}(\Theta^\mathsf{r}(\partial_u,\partial_x)) - \Theta^\mathsf{r}(\mathcal{L}_{\vec{H}}(\partial_u,\partial_x))\\
& = (-1)^{k(n-k)} \left[\vec{H}( \Theta(\partial_v,\partial_u,\partial_x)) -\Theta(\partial_v,\mathcal{L}_{\vec{H}}(\partial_u,\partial_x))\right] \\
& = (-1)^{k(n-k)} \left[(\mathcal{L}_{\vec{H}}\Theta)(\partial_v,\partial_u,\partial_x) + \Theta(\mathcal{L}_{\vec{H}}(\partial_v),\partial_u,\partial_x)\right] \\
& =\Theta(\partial_u,\mathcal{L}_{\vec{H}}(\partial_v),\partial_x), \label{eq:derthetared}
\end{align}
where, in the last step, we used that $\mathcal{L}_{\vec{H}}\Theta = 0$.
Now observe that, for $j=1,\ldots,n-k$,
\begin{align}
[\vec{H},\partial_{v_j}] & = \sum_{i=1}^k [u_i \vec{u}_i,\partial_{v_j}] = \sum_{i=1}^k u_i[\vec{u}_i, \partial_{v_j}] \\
& = -\sum_{i=1}^k \sum_{\ell=1}^{k} u_i \partial_{v_j} \{u_i,u_\ell\} \partial_{u_\ell} -\sum_{i=1}^k \sum_{\ell=1}^{n-k} u_i \partial_{v_j} \{u_i,v_\ell\} \partial_{v_\ell}. \label{eq:derpartialv}
\end{align}
Plugging~\eqref{eq:derpartialv} in~\eqref{eq:derthetared}, and using complete skew-symmetry, we obtain the statement.
\end{proof}

\section{Improving estimates through reduction}\label{sec:closed-geodesics}

We sketch a strategy to improve the lower bound for the (sub-)Lapla\-cian of Proposition~\ref{p:lowerbound-intro}, when the latter is trivial, i.e. when $L=+\infty$. Similar considerations hold also for isoperimetric-type inequalities. 

Let $M$ be a (sub-)Riemannian manifold with boundary $\partial M$, and assume that a reduced bundle $U^*M^\mathsf{r}$ has been found in such a way that the reduced Santal\'o formula, and all its consequences, hold.  Assume that there exists a reduced geodesic that never hits the boundary of $M$. This happens if there exists a covector $\lambda \in U^*M^\mathsf{r}$ such that $L(\lambda) = +\infty$. This phenomenon occurs already in the \emph{Riemannian} case, where no reduction is required, and in particular if $M$ contains closed Riemannian geodesics. For example, consider the small rotationally symmetric neighborhood
\[
M = \{(\theta,\phi) \mid \pi/2- \varepsilon \leq \theta \leq \pi/2 + \varepsilon\} \subset \mathbb{S}^{2},
\]
of the equator $\theta = \pi/2$ of the two dimensional round sphere, equipped with the standard measure. In this cases, the lower bound of Proposition~\ref{p:lowerbound-intro} for the Dirichlet spectrum on $M$ is trivial, since $L = \sup_{\lambda \in U^*M} L(\lambda) = +\infty$. Nevertheless the reduction procedure can be still applied to circumvent this problem, as we now sketch for the spherical band above.

The idea is to define a set of reduced geodesics by considering only those which are normal to the boundary $\partial M$. More precisely, we let
\begin{equation}
\distr := \spn\{ \partial_\theta\}, \qquad \ver := \spn \{ \partial_\phi \},
\end{equation}
and we set $T^* M^\mathsf{r} = \ver^\perp = \spn\{d \theta \}$. The restriction of $g$ to $\distr$ induces a sub-Riemannian structure (which does not satisfy the bracket-generating condition, but this is inconsequential here). The Dirichlet energy of $(M,\distr,g)$ is not greater than the one of the original Riemannian structure. Hence a lower bound for the first eigenvalue of the sub-Laplacian of $(M,\distr,g|_{\distr})$ yields a lower bound for the Laplace-Beltrami operator on $M$. 

Geodesics with $\lambda \in U^*M^\mathsf{r}$ cross the spherical band longitudinally, and $L(\lambda) = 2\varepsilon$. Both $(\mathbf{H1})$ and $(\mathbf{H2})$ are verified, and thus we obtain from Proposition~\ref{p:lowerbound-intro} the sharp estimate
\begin{equation}
\lambda_1(M) \geq \frac{\pi^2}{(2\varepsilon)^2}.
\end{equation}

This construction highlights the fact that the reduction procedure can be used in both the Riemannian and sub-Riemannian case to improve estimates such as the one of Proposition~\ref{p:lowerbound-intro}, when the geometry of the problem is quite explicit. The general philosophy is that the smaller is the set of reduced geodesics, the better is the bound in Proposition~\ref{p:lowerbound-intro}.

\section*{Acknowledgments} 
We are grateful to the organizers of the Trimester \href{http://www.cmap.polytechnique.fr/subriemannian/}{\emph{Geometry, Analysis and Dynamics on Sub-Riemannian Manifolds}}, whose stimulating atmosphere fostered this collaboration, and the Institut Henri Poincar\'e, Paris, where most of this research has been carried out. We also thank the Isaac Newton Institute for Mathematical Sciences, Cambridge, for the support and  hospitality during the programme \href{http://www.newton.ac.uk/event/pep}{\emph{Periodic and Ergodic Spectral Problems}}. The last named author is grateful for the support and hospitality of the Erwin Schr\"odinger Institute of Mathematical Physics in Vienna during the programme \href{http://www.esi.ac.at/activities/events/2015/modern-theory-of-wave-equations}{\emph{Modern theory of wave equations}}, where he stayed when this article was completed. We thank F. Baudoin and J. Wang for useful discussions about the spectrum of the sub-Laplacian on the quaternionic Hopf fibration and L. Parnovski for useful comments on the spectral theoretic part. We thank Z. Balogh for useful discussions on the negligibility of the set of characteristic points.

\bigskip

This research has  been supported by the European Research Council, ERC StG 2009 ``GeCoMethods'', contract n. 239748, by the iCODE institute (research project of the Idex Paris-Saclay), and by the SMAI project ``BOUM''. The first and second authors were partially supported by the Grant ANR-15-CE40-0018 of the ANR. This research benefited from the support of the ``FMJH Program Gaspard Monge in optimization and operation research'' and from the support to this program from EDF.

\bibliographystyle{abbrv-JDG}
\bibliography{sant-biblio}

\end{document}